\documentclass[11pt,reqno]{amsart}

\usepackage{latexsym}
\usepackage{amssymb, tikz}
\usepackage{mathrsfs}
\usepackage{mathabx}

\usepackage{color}

\usepackage{amsthm}
\theoremstyle{plain}
\usepackage{MathShorthand}%

\newtheorem*{theorem*}{Theorem}

\renewcommand{\epsilon}{\varepsilon}

\newcommand{\R}{{\mathbb R}}

\newcommand{\Z}{{\mathbb Z}}

\newcommand{\h}{\hbar}

\renewcommand{\phi}{\varphi}

\newtheorem{theo}[equation]{{\sc Theorem}}

\newtheorem{cor}[equation]{{\sc Corollary}}

\newtheorem{conj}[equation]{{\sc Conjecture}}

\newtheorem{lem}[equation]{{\sc Lemma}}

\newtheorem{prop}[equation]{{\sc Proposition}}

\theoremstyle{definition}
\newtheorem{defn}[equation]{{\sc Definition}}

\theoremstyle{remark}
\newtheorem{rem}[equation]{Remark}
\theoremstyle{assumption}

\title[Small perturbations of spheres]{Laplace-Beltrami spectrum of ellipsoids that are close to spheres and analytic perturbation theory}

\author{Suresh Eswarathasan}
\address{Dalhousie University \\
Department of Mathematics $\&$ Statistics, Chase Building Rm 316 \\
Coburg Road \\
Halifax, Nova Scotia \\
Canada}
\email{sr766936@dal.ca}

\author{Theodore Kolokolnikov}
\address{Dalhousie University \\
Department of Mathematics $\&$ Statistics, Chase Building Rm \\
Coburg Road \\
Halifax, Nova Scotia \\
Canada}
\email{tkolokol@gmail.com}

\begin{document}

\begin{abstract}
We study the spectrum of the Laplace Beltrami operator on ellipsoids.
For ellipsoids that are close to the sphere, we use analytic perturbation theory
to estimate the eigenvalues up to two orders. We show that for biaxial
ellipsoids sufficiently close to the sphere, the first $L^2$ eigenvalues have
multiplicity at most two, and characterize those that are simple. For the
triaxial ellipsoids sufficiently close to the sphere that are not biaxial, we
show that at least the first sixteen eigenvalues are all simple.

We also give the results of various numerical experiments, including comparisons
to our results from the analytic perturbation theory, and approximations for the
eigenvalues of ellipsoids that degenerate into infinite cylinders or
two-dimensional disks. We propose a conjecture on the exact number of nodal
domains of near-sphere ellipsoids.  \end{abstract}

\maketitle

\section{Introduction}

This article aims to compute the Laplace-Beltrami spectrum (and its
multiplicities) of a class of 2-dimensional ellipsoids in $\mathbb{R}^3$
through analytic perturbation theory \cite{R69}, more specifically 
eigenvalue perturbations \cite{Hin95}.  
We describe our main results first before providing
motivations and a discussion of related results.

\subsection{Main results}

Let us give the definition for the main object of study in our article:
\begin{defn} \label{defn:E_abc}
Let $a,b,c>0$.  We denote by $E_{a,b,c} \subset \R^3$ the ellipsoid given by
\begin{equation*}
 \left\{ (x,y,z) \in \mathbb{R}^3 \, | \, \frac{x^2}{a^2} + \frac{y^2}{b^2} + \frac{z^2}{c^2} = 1 \right\}.
\end{equation*}
In other words, $E_{a,b,c}$ is an ellipsoid with axes $a,b$, and $c$.
\end{defn}

{Let $-\Delta_g$ be the positive Laplace-Beltrami operator on $E_{a,b,c}$ and consider the corresponding eigenvalue problem $-\Delta_g
\phi_{\Lambda} = \Lambda \phi_{\Lambda}$.    Recall that on the sphere $S^2 = E_{1,1,1}$, the eigenvalues are $\Lambda = l(l+1)$ having multiplicity
$2l+1$, with $l \in \mathbb{Z}_{\ge0}$. We are now in a position to state main results}:

\begin{theo} \label{thm:biaxial}
Let $L \in \mathbb{N}$ and $\alpha, \beta \in \mathbb{R}$ with at least one being non-zero.  Consider the biaxial ellipsoid $E_{a,a,b}$ where $%
a=1+\varepsilon \alpha ,\ b=1+\varepsilon \beta $ where $\epsilon \in \mathbb{R}^+$ and $g_{\epsilon}$ the metric from $\mathbb{R}^3$ restricted to $E_{a,a,b}$.

Then there exists $\epsilon_0(\alpha, \beta,L)$ such that for all $\epsilon < \epsilon_0$ and $\Lambda \in {\rm{spec}}(-\Delta_g) \cap [0,L(L+1)]$, we have
\begin{equation*}
\Lambda =l\left( l+1\right) +\varepsilon \Lambda _{1}+O(\varepsilon ^{2})
\end{equation*}%
for $ l=0,1,2,\ldots L$ and $m=-l,\ldots ,l$ with $\Lambda _{1}$ being given by the explicit formula
\begin{equation} \label{eqn:Lambda1biaxial}
\Lambda _{1}= -2\alpha
l\left( l+1\right)  + \left( \alpha -\beta \right) \frac{2l\left( l+1\right) }{\left(2l+3\right) \left( 2l-1\right) }\left( 2l^{2}-2m^{2}+2l-1\right).
\end{equation}
Moreover, each $\Lambda$ has multiplicity two except for those whose expansion has $m=0$, which in this case corresponds to multiplicity one.
\end{theo}

\begin{theo} \label{thm:triaxial}
Let $l \in \mathbb{N}$ and $\alpha, \beta, \gamma \in \mathbb{R}$ be given with at least one being non-zero.  Consider the triaxial ellipsoid $E_{a,b,c}$ where $%
a=1+ \alpha \varepsilon ,\ b=1+\beta \varepsilon$ and $c = 1 + \gamma \epsilon$ and $g_{\epsilon}$ the metric from $\mathbb{R}^3$ restricted to $E_{a,b,c}$.

Then there exists $\epsilon_0(\alpha, \beta, \gamma, l)$ such that for all
$0<\epsilon < \epsilon_0$ and $\Lambda \in {\rm{spec}}(-\Delta_g) \cap [l(l+1) - 2l, l(l+1) + 2l]$, we have
\begin{equation} \label{eqn:triax_exp}
\Lambda=l\left(  l+1\right)  + \Lambda_{1} \, \varepsilon + \mathcal{O}(\epsilon^2)
\end{equation}
where $\Lambda_{1}$ is an eigenvalue of a $(2l+1) \times (2l+1)$ matrix and
whose entries yield explicit formulas in $l, \alpha, \beta,$ and $\gamma$.
Thus, given $L \in \mathbb{N}$, there exists $\epsilon_0(\alpha, \beta,
\gamma, L)$ such that the expansion (\ref{eqn:triax_exp}) holds for all
$\Lambda \in {\rm{spec}}(-\Delta_g) \cap [0, L(L+1)]$.  Lastly, for
$l=1,2,3$ in particular, there exists $\epsilon_0$ such that ${\rm{spec}}(-\Delta_g) \cap [l(l+1) - 2l, l(l+1) + 2l]$ contains only simple eigenvalues for all $\epsilon < \epsilon_0$.
\end{theo}

For these explicit formulas pertaining to $\Lambda_1$, see Proposition \ref{prop:triax_entries} in Section \ref{sect:triax}.

\subsection{Some motivations}

From the point of view of classical mechanics, ellipsoids $E_{a,b,c}$ form one of the oldest known examples of integrable systems, themselves holding a venerable place in the subject.  Their quantum analogues have been intensively studied in the last forty years, with numerous contributions arising from a beautiful mixture of symplectic geometry and WKB approximations, the connection being exploited by microlocal/semiclassical analysis.  We survey some related and microlocally-oriented results in the next section.

We emphasize that the spectrum of geometric spaces with large symmetry groups has been explicitly computed \cite{Ter} with a partial list being compact rank-one symmetric spaces (CROSSes), certain projective spaces, Steifel manifolds, and Grassmannians. However, it appears that not much is known for manifolds lacking large symmetry groups like biaxial and triaxial ellipsoids let alone their multiplicities.  While microlocal analysis has addressed the approximation of the eigenvalues in the semiclassical limit, we have not found any literature providing bounds on multiplicity.

The study of Laplace's equation on $\R^3$ using ellipsoidal-type coordinate systems is well-developed and is centered around the analysis of the Lam\'e equation.  In fact, the word ``ellipsoidal harmonics" has been attached to a variety of families of functions including for eigenfunctions on $E_{a,b,c}$.  See the treatise of Dassios \cite{Dass} for a survey and in particular Chapter 4.4 for a brief treatment on product-form eigenfunctions on $E_{a,b,c}$.

We make note that the \textit{use of analytic perturbation theory allows for both accurate approximations (as confirmed by our numerics in Section \ref{sect:biaxial_numerics}) and multiplicity calculation}.  Furthermore, it provides us the opportunity to bypass the use of Bohr-Sommerfeld quantization rules and the computation of subprincipal symbols, per microlocal analysis, which are two highly powerful but technical concepts.  In this vein, the contents of our article appear to be novel.

The question of what is $\rm{spec}(-\Delta_g)$ has played a prominent role in recent years in data analysis. For example, in \cite{RWP06} eigenvalues of the Laplace–Beltrami operator were used to extract ``fingerprints" which characterize surfaces and solid objects. In \cite{BN03, CL06}, these eigenvalues (and their corresponding eigenfunctions) were used for dimensionality reduction and data representation.

We close our motivations section by briefly discussing the concept of shape DNA in shape matching.  By ``shape DNA", we mean the first $N$ elements of $\rm{spec}(-\Delta_g)$ for $(M,g)$.  Our results can be encompassed as the computation of this shape DNA for ellipsoids that are close to spheres.  Note that shape DNA plays a crucial role in the representation of data sets, itself being usefuk in copyright protection and database retrieval.  For more applications, see \cite{RWP06}.

\subsection{Related results from semiclassics}

There are a number of relevant results from the semiclassical analysis literature that require discussion.  We begin with the work of Pankratova \cite{P70}.  In this article, the author uses a special ellipsoidal coordinate system and the so-called parabolic-equation method (in the spirit of Babich and Lazutkin) to compute high-frequency asymptotics for eigenvalues arising only from product-form eigenfunctions of $-\Delta_g$.

Three works of greatest relevance to Theorems \ref{thm:biaxial} and \ref{prop:triax_entries} are those of Sj\"ostrand \cite{Sj92}, Colin de Verd\`ere-Parisse \cite{CdVP99}, and Toth \cite{T95}, each of which we describe in detail.  First, the work of Toth \cite{T95} not only proves the quantum integrability of $E_{a,b,c}$ (i.e. the existence of a second quantum Hamiltonian on $E_{a,b,c}$ that commutes with $\Delta_g$) but also formulates an interesting conjecture: the joint spectrum for these quantum Hamiltonians is encoded by a second-order complex ODE with automorphic boundary conditions.

An asymptotic description of eigenvalues for semiclassical Schr\"odinger operators whose potentials satisfy a non-degeneracy condition is given in the work of Sj\"ostrand \cite{Sj92}.  In the case of biaxial 2-dimensional ellipsoids, one can reduce to a Mathieu-type operator $A(\h)$ on the non-$S^1$-invariant axis and then use quantum Birkhoff normal forms to read off formulas for the energies in a fixed window (in our setting, this corresponds to the low-energy spectrum with the constraints that $m^2 \approx l^2$) but to order $\h$.  This leads us to the more geometric work of Colin de Verdi\`ere-Parisse.

The articles \cite{CdVVN03, CdVP99} investigate Bohr-Sommerfeld quantization rules in the presence of singularities generated by select classes of quantum Hamiltonians.  The spectrum of ellipsoids is studied as an application by Colin de Verdi\`ere-Vu Ngoc and depends strongly on some previous work of Colin de Verdi\`ere-Parisse in one dimension.  The work \cite{CdVP99} determines that the energies of a certain class of 1-dimensional semiclassical Schr\"odinger operators $P(\h)$ in a fixed window can be  explicitly deduced from solving for the coefficients $a_j$ in the equation $e^{\frac{\sum_j a_j \h^j}{\h}} = 1$ for small enough $h$.  Following this reasoning in the case of biaxial 2-dimensional ellipsoids should allow one to reproduce the multiplicity information given in Theorem \ref{thm:biaxial}.  To perform this calculation however, it appears one needs to push the quantum Birkhoff normal forms to greater precision via the calculation of $P(\h)$'s subprincipal terms.  In some sense, our work proceeds in this direction albeit through the lens of analytic perturbation theory.  This is one avenue in which our Theorems \ref{thm:biaxial} and \ref{thm:triaxial} are new.

{ In closing this discussion, we point the connections between our method and high-frequency quasimode constructions in semiclassical
analysis.  If we write out the Laplace-Beltrami operator on $E_{a,b,c}$ as $-\h^2 \left( \Delta_g + \epsilon A_1 + \mathcal{O}(\epsilon^2) \right)$, it becomes
clear that we are utilizing an additional small parameter $\epsilon$ whilst bounding $\h$ away from zero to write quasimodes on $S^2$.  In fact, we are
computing the quasifrequencies for $-\Delta_g$ but not in a high-frequency regime. This naturally results in our Theorems \ref{thm:biaxial} and
\ref{thm:triaxial} not being descriptive of high frequencies but at the upshot of being descriptive for multiplicities.  }

\subsection{Outline of the paper}

{Our main tool is the theory of perturbations, a sharply defined set of ideas that is described for instance in the classic applied mathematics
text of Hinch \cite{Hin95} and in a more pure, theoretical fashion in the treatise of Kato \cite{K}.
While there exists a number of sources
from pure mathematics rigorizing asymptotics, including \cite{K}, there appear to be much fewer sources demonstrating the analyticity of eigenvalues
for analytic families of metrics.  In this paper, we utilize a combination of results from Rellich's perturbation theory notes (unfortuntaely, now
discontinued) from the Courant Mathematical Institute \cite{R69} and an article of Bando-Urakawa \cite{BU83} on eigenvalues for certain families of
Laplace-Beltrami operators.  In fact, for the benefit of easier reading, Section \ref{sect:analyt_pert} is dedicated to an appreciable reproduction of useful
results from \cite{BU83} along with some alternate proofs coming from Rellich's Courant notes.

Once we have explained our theoretical tools, particularly in Theorem \ref{thm:berger_lem}, Sections \ref{sect:triax} and \ref{sect:triax} are dedicated to explicit calculations with the coordinate representations of $-\Delta_{g}$ on $E_{a,b,c}$ and the ultraspherical harmonic basis on $L^2(S^2)$.  Section \ref{sect:biax} utilizes the symmetries of $E_{a,b,b}$ to reduce our eigenvalue calculation problem to one for analytic families of ordinary differential equations, allowing us to apply the theory of Sturm-Liouville equations as well as make deductions on multiplicities.  The triaxial ellipsoid $E_{a,b,c}$ is the most difficult computationally, so Section \ref{sect:triax} and the appendix are focused in this direction.  The lack of rotation-invariance (in other words, an invariant $S^1$-action on $E_{a,b,c}$) obstruct most simplifications hence requiring a more in-depth analysis.

Our final section, namely Section \ref{sect:numerics}, is focused on verifying the accuracy of our analytic methods via numerics.  A combination of MATLAB
calculations as well as Laplace-Beltrami eigenfunction approximations on surfaces, as generated by code of Macdonald-Brandman-Rooth \cite{Mac}, are provided
for comparisons: these demonstrate that our analytic results are in fact accurate up to a designated, yet still high, order.  We also provide some simulations
that address the shapes of regions where eigenfunctions are non-zero (which go by the moniker of ``nodal domains" in the spectral geometry community) on
different ellipsoids $E_{a,b,c}$.}

\section{Analytic Perturbation Theory} \label{sect:analyt_pert}

Let $M$ be a compact, $n$-dimensional, smooth manifold without boundary.  Let $S(M)$ be the space of all $C^{\infty}$ symmetric covariant 2-tensors on $M$ and $\mathscr{M}$ the set of all $C^{\infty}$ Riemannian metrics on $M$.  Following the texts \cite{E71, GG}, we can put a Frechet norm on $S(M)$. Using this fact, \cite[Proposition 1.2]{BU83} gives a metric $\rho$ on $\mathscr{M}$ which will play a role in the statement of our following theorem although its precise form is not needed.

\begin{theo}[Berger's Lemma] \label{thm:berger_lem}
For $g \in \mathscr{M}$ and $h_i \in S(M)$ fixed for $i=1,\dots, N$, let $g(\epsilon) = g + \sum_{i=1}^N \epsilon^i h_i$ where $\epsilon < \epsilon_0(M,\{h\}_i)$.  Let $\Lambda$ be an eigenvalue of $-\Delta_g$ of multiplicity $l$.  Then $\mbox{spec}(-\Delta_{g(\epsilon)})$ consists of elements that have an analytic dependence in $\epsilon$ in the following way: Given $\Lambda$, there exists $\epsilon_1(M, \epsilon)$ along with $\Lambda_m(\epsilon) \in \R$ and $\psi_m(\epsilon) \in C^{\infty}(M)$, for $m \in \{1,\dots, l \}$, such that
\begin{enumerate}
\item $\Lambda_m$ and $\psi_m$ depend real-analytically on $\epsilon < \epsilon_1$, uniformly for each $m \in \{ 1, \dots, l \}$,
\item $\Lambda_j(0) = \Lambda$, for $m \in \{ 1, \dots, l \}$, and
\item $\{\psi_m(\epsilon)\}_{m=1}^l$ is orthonormal with respect to the inner product $\langle , \rangle_{g_{\epsilon}}$
\end{enumerate}
\end{theo}

This ``lemma" is originally due to Berger \cite{B73} however some gaps needed to be resolved and were filled by Bando-Urakawa \cite{BU83}.  Their own proofs though, albeit terse, heavily relied on various facts from the perturbation theory of eigenvalue problems, so we reproduce a sufficient number of arguments due to Rellich \cite{R69} for the following two reasons: 1) sake of completeness and 2) to re-illustrate the beautiful blend of ideas and formulas presented by Rellich.

The proof of Theorem \ref{thm:berger_lem} actually follows as an immediate consequence of a slightly more general result. However, we first give a necessary definition:

\begin{defn}
A family of metric $\{ g_{\epsilon} \}_{\epsilon} \subset \mathscr{M}$ depends \textit{real-analytically on $\epsilon$} if there exists a family $\{g_i\}_{i=0}^{\infty} \subset S(M)$ and an $\epsilon_0(M,\{g\}_i)$ such that $\sum_{i=0}^{\infty} \epsilon^i g_i$ converges to $g$ in the metric topology of $\mathscr{M}$, for all $\epsilon < \epsilon_0$.
\end{defn}

\begin{theo}(cf. \cite[Theorem 1]{BU83}]) \label{thm:main_tool}
Let $g_{\epsilon} \in \mathscr{M}$ be a one-parameter family of metrics depending real-analytically on $\epsilon < \epsilon_0$ with respect to the metric $\rho$ on $\mathscr{M}$, for some $\epsilon_0(M)>0$.  Let $\Lambda$ be an eigenvalue of $-\Delta_g$ of multiplicity $l$.   Then the spectrum $\mbox{spec}(-\Delta_{g(\epsilon)})$ consists of elements that have an analytic dependence in $\epsilon$ in the following way: for $\Lambda$,  there exists $\epsilon_1(M, \epsilon_0,\Lambda)$ along with $\Lambda_m(\epsilon) \in \R$ and $\psi_m(\epsilon) \in C^{\infty}(M) \cap L^2(M, g_{\epsilon})$, for $m \in \{1,\dots, l \}$, such that
\begin{enumerate}
\item $\Lambda_m$ and $\psi_m$ depend real-analytically on $\epsilon < \epsilon_0$ (with respect to their corresponding topologies), for each $m \in \{ 1, \dots, l \}$,
\item $\Lambda_j(0) = \Lambda$ and $\psi_m(0)$ is in the $-\Delta_g$-eigenspace associated to $\Lambda$, for $m \in \{ 1, \dots, l \}$, and
\item $\{\psi_m(\epsilon)\}_{m=1}^l$ is orthonormal with respect to the inner product $\langle , \rangle_{g_{\epsilon}}$
\end{enumerate}
\end{theo}

\subsection{Main tools and Rellich's Theorem}

The proof of this Theorem \ref{thm:main_tool} hinges upon the aforementioned robust and clever result of Rellich \cite{R69} and an auxilliary lemma about linear differential operators whose coefficients have an analytic dependence on a small parameter.  First, we start with a definition for the notion of real-analytic operators:

\begin{defn}[Real-analytic families of operators] \label{def:analytic}
For $\epsilon >0$, let $A(\epsilon) \in \mathscr{L}(H^{s_1}(M), \newline H^{s_0}(M))$, the Banach space of bounded operators from $H^{s_1}(M)$ to $H^{s_0}(M)$.  We say $A(\epsilon)$ is \textit{real-analytic in $\epsilon$} if there exists a $A_i \in \mathscr{L}(H^{s_1},H^{s_0})$ with the property that $A(\epsilon) = \sum_i \epsilon^i A_i$  and a  sequence of constants $\{a_i\}_i$ where $\|A_i\| \leq a_i$ such that
\begin{equation*}
\| A(\epsilon) \| \leq \sum_n \epsilon^n a_n < \infty.
\end{equation*}
\end{defn}

In fact, the definition goes both ways: starting off with the series expansion and finiteness of its norm, that $\mathcal{L}$ is Banach gives us that $A(\epsilon)$ is in fact an element of $\mathscr{L}(H^{s_1},H^{s_0})$.  We now give a technical lemma that is useful for analysis in coordinate charts:

\begin{lem}\label{lem:analytic_DO}
Let $U \subset \R^n$ be a coordinate chart on $M$.  Let $L_{\epsilon}$ be family of differential operators on $M$ which can locally be expressed in $U$ as
\begin{equation*}
L_{\epsilon} = \sum_{|\alpha| \leq m} a_{\alpha}(\epsilon, x) D_x^{\alpha}
\end{equation*}
where each $a_{\alpha}$ has a real-analytic dependence on $\epsilon < \epsilon_0(U)$, uniformly for $x \in U$.  Then the family of bounded operators $L_{\epsilon}: H^{m}(M) \rightarrow L^2(M)$ is real-analytic.
\end{lem}

\begin{proof}
We leave the proof as an easy exercise for the interested reader.
\end{proof}

It is important to note that in local coordinates $(x_1, \dots, x_n)$ on $M$, the coefficients of the Laplace-Beltrami operator are simply products of functions which are themselves analytic in $\epsilon$ thanks to our analyticity assumption on $g_{\epsilon}$ (this assumption itself implying analyticity for the coefficients of $g^{-1}$, which appear in the local coordinate expressions of $\Delta_{g_{\epsilon}}$, thanks to the analyticity of $\det^{-1}(g)$ and the adjugate matrix of $g$.)

Finally, we arrive at our main technical results in the theoretical portion of this article:

\begin{theo}[Rellich's Theorem] \label{thm:rellich}
Let $I$ be an interval containing 0.  Let $s_1 > s_0 \geq 0$ be integers.  Let $A_{\epsilon}$ be a real-analytic family of bounded operators mapping from $H^{s_1}$ to $H^{s_0}$ with $A_0 =: A$.  Assume that
\begin{enumerate}
\item each operator $A_{\epsilon}$, $\epsilon \in I$, is self-adjoint with domain $H^{s_1}$ but with respect to the inner product $\langle \, , \, \rangle_{s_0}$.  In other words, $A_{\epsilon}$ is a densely defined unbounded operator on $H^{s_0}(M)$ and has $D(A_{\epsilon}) = D(A_{\epsilon}^*)$,
\item $A_0$ is a positive operator on its domain, and
\item $\Lambda$ is an eigenvalue of $A_0$ with multiplicity $l$ that is also isolated in the spectrum, that is there exists $\delta(\Lambda) > 0$ such that $spec(A) \cap [-\delta + \Lambda, \delta + \Lambda] = \{ \Lambda\}$.
\end{enumerate}

Then there exists $I' \subset I$ containing 0,  $l$ real-analytic families of eigenvalues $\{ \Lambda_m(\epsilon) \}_{m=1}^l$ and eigenvectors $\{ \psi_m(\epsilon) \}_{m=1}^l$ of $A_{\epsilon}$ for $\epsilon \in I'$ such that
\begin{itemize}
\item $\Lambda_j(0) = \Lambda$, for $m \in \{ 1, \dots, l \}$,
\item $\{\psi_m(\epsilon)\}_{m=1}^l$ is orthonormal with respect to the inner product $\langle , \rangle_{s_0}$ for all $\epsilon \in I'$, and moreover,
\item given any $d_1, d_2 < \delta$, there exists $I'' \subset I'$ such that for all $\epsilon \in I''$, $\mbox{spec}(A_{\epsilon}) \cap  [-d_1 + \Lambda, d_2 + \Lambda] = \{ \Lambda_m(\epsilon) \}_{m=1}^l$.
\end{itemize}
\end{theo}

A similar statement can be found in the classic texts of Kato \cite{K} and Riesz-Nagy \cite{RN}. Note also that a priori, our intervals $I'$ and $I''$ depend on $\Lambda$ therefore making this result inherently non-uniform across the entirety of $\mbox{spec}(A_0)$.

In fact, by applying Theorem \ref{thm:rellich} to each element of the spectrum of $A$ below a fixed threshold $L$ say, and carefully choosing the intervals $I''$ we then immediately have the
\begin{cor} \label{cor:all_analyt}
Given $L>0$, there exists $\epsilon_0(L)$ such that $\mbox{spec}(A(\epsilon)) \cap [0, L]$ consists entirely of analytic eigenvalues as described in Theorem \ref{thm:rellich}.
\end{cor}

\begin{proof}[Proof of Theorem \ref{thm:main_tool} using Lemma \ref{lem:analytic_DO} and Theorem \ref{thm:rellich}]
The proof is almost immediate after considering the following well-known isometry between $L^2(M,g_{\epsilon})$ and $L^2(M,g)$ which we give in local-coordinate form:
\begin{equation*}
U_{\epsilon}(f)(x) = \sqrt{\frac{\det g(x)}{\det g_{\epsilon}(x)}} \, f(x).
\end{equation*}
Thanks to the analyticity of $\sqrt{\frac{\det g(x)}{\det g_{\epsilon}(x)}}$, it follows immediately that $U_{\epsilon} \Delta_{g_{\epsilon}} U_{\epsilon}^{-1} + I$ satisfies the hypotheses of Lemma \ref{lem:analytic_DO} and is a bounded family of operators from $H^2(M, g_{\epsilon})$ to $L^2(M, g)$ where $s_1=2$ and $s_0=0$. Furthermore, these operators are densely defined on $L^2(M)$, self-adjoint, and positive, thus satisfying the hypotheses of Theorem \ref{thm:rellich}.

We conclude the proof by noting that for $\tilde{\psi}_{m}(\epsilon)$ an eigenvector of $U_{\epsilon} \Delta_{g_{\epsilon}} U_{\epsilon}^{-1} + I$ as per the conclusions of Theorem \ref{thm:rellich}, $\psi_m(\epsilon) := U_{\epsilon}^{-1} \tilde{\psi}_{m}(\epsilon)$ gives us the desired eigenvectors corresponding to $g_{\epsilon}$. Hence,
\begin{equation*}
\Lambda_m(\epsilon) = \langle \left( U_{\epsilon} \Delta_{g_{\epsilon}} U_{\epsilon}^{-1} + I \right) \tilde{\psi}_{m}(\epsilon), \tilde{\psi}_{m}(\epsilon) \rangle_{g_0}
\end{equation*}
which itself admits an power series expansion therefore verifying the analyticity.  The last step is to just shift the spectrum by -1.
\end{proof}

\subsection{Proof of Theorem \ref{thm:rellich}: Some technical statements}

The idea behind the proof of Rellich's Theorem is both natural and computational in nature, however there are a number of moving parts that we must carefully identify in a top-down format. Throughout this section, we consider assumptions (1)-(3) in the statement of Theorem \ref{thm:rellich}.

For the sake of simplicity, as our operators of interest are themselves Laplace-Beltrami operators corresponding to perturbed metrics $g_{\epsilon}$, we assume that $A_{\epsilon}$ admits a discrete, non-negative spectrum and that $\mbox{Dom}(A_{\epsilon}) = \mbox{Dom}(A_0) = H^{2}(M, dV_{g_0})$.  Now, set $B_{\epsilon} := A_0 - A_{\epsilon}$ and $\mu(\epsilon) := \Lambda - \Lambda(\epsilon)$ where $A_{\epsilon} \psi(\epsilon) = \Lambda(\epsilon) \psi(\epsilon)$.  This leads us to the following series of simple lemmas:

\begin{lem}[Restriction and matrix identity] \label{lem:restriction_identity}
Let $\phi_{\Lambda}$ be a eigenvector of eigenvalue $\Lambda$ and $\psi(\epsilon)$ an eigenvector of $A_{\epsilon}$.  We have
\begin{align*}
(A_0 - \Lambda) \psi(\epsilon) = (B_{\epsilon} - \mu(\epsilon)) \psi(\epsilon)
\end{align*}
and hence $\langle (A - \Lambda) \phi_{\Lambda}, \psi(\epsilon) \rangle = \langle \phi_{\Lambda}, (B_{\epsilon} - \mu(\epsilon)) \psi(\epsilon) \rangle = 0 $ by self-adjointness with domain $H^{2}$.
\end{lem}

\begin{lem}[Pseudo inverses] \label{lem:pseudo_inv}
Let $A$ be our self-adjoint, unbounded operator on $L^2(M)$ and $\delta$ be as in the hypotheses Theorem \ref{thm:rellich}.  There exists $R \in \mathscr{L}(L^2) \cap \mathscr{L}(H^{2})$ such that
\begin{itemize}
\item $R \, \Pi_{E_{\Lambda}} = 0$.
\item $R \, (A - \Lambda) =  (A - \Lambda) \, R = I - \Pi_{E_{\Lambda}} $
\end{itemize}
\end{lem}

\begin{proof}
We invoke spectral calculus for unbounded operators and denote by $E_{\sigma}$ the spectral measure for $A$ (which is in our case discrete).  Then we can set
\begin{equation*}
R = \int_0^{\Lambda - \delta} \frac{1}{\sigma - \Lambda} \, dE_{\sigma} + \int_{\Lambda + \delta}^{\infty} \frac{1}{\sigma - \Lambda} \, dE_{\sigma}.
\end{equation*}
The boundedness follows from spectral multiplier taking values less than $\delta^{-1}$ on the spectrum.
\end{proof}

Notice that for $\psi(\epsilon)$ an eigenfunction of $A_{\epsilon}$, we have the orthogonal decomposition $\psi(\epsilon) = \overline{\psi}(t) + R(B_{\epsilon} - \mu(\epsilon)) \psi(\epsilon)$ where $\overline{\psi}(t) \in E_{\Lambda}$; this follows immediately from a combination of Lemma \ref{lem:restriction_identity} and Lemma \ref{lem:pseudo_inv}.  Iterating this expression, if we have a convergent operator series, allows us to express $\psi(\epsilon)$ completely in terms of data coming from $E_{\Lambda}$.  This notion motivates the main idea of Rellich's proof as seen through the following natural generalization of this series.

\begin{lem}[Neumann series for $\psi(\epsilon)$] \label{lem:neumann}
Let $\mu$ be a free parameter and $B_{\epsilon} = \sum_{i=1} \epsilon^i A_i$ be analytic in $\epsilon$ in the sense of Definition \ref{def:analytic}, with each of its terms $A_i$ having norm $\|A_i\|_{op} \leq c_0$ where $c_0>0$ is fixed.  Consider $\overline{\psi} \in E_{\Lambda}$ and set $S(\epsilon) = R \circ (B_{\epsilon} - \mu)$ with $\|R\|=r_0$.

Then there exists $\epsilon_0, \mu_0$ small enough such that for $|\epsilon| < \epsilon_0$ and $|\mu| < \mu_0$ , we have that
\begin{equation} \label{eqn:putative_eigen}
\psi(\epsilon) = \sum_{n=0}^{\infty} S(\epsilon)^n (\overline{\psi}(\epsilon))
\end{equation}
exists in $L^{2}$.
\end{lem}

This element $\psi(\epsilon)$ will be shown to be our desired eigenfunction for $A_{\epsilon}$.  It should be noted that the expression for this putative eigenfunction involves only information from a fixed eigenspace for $A_0$, namely $E_{\Lambda}$.  Also, the reason for having only a single number $c_0$ bounding the norms of our operator-valued coefficients is because we only deal with 2nd-order differential operators on $M$, therefore yielding only 2nd-order operators with uniformly bounded real coefficients.

\begin{proof}[Proof of Lemma \ref{lem:neumann}]
We only need to verify that $\sum_{n=0}^{\infty} S(t)^n$ has a small norm for $t$ sufficiently small.  Hence, we must further bound $\sum_{n=1}^{\infty} \|R \circ (B_{\epsilon} - \mu ) \|^n$, which in turn leads us to bounding  $\|R \circ (B_{\epsilon} - \mu(\epsilon)) \|^n \leq \|R\|^n \cdot \| (B_{\epsilon} - \mu(\epsilon)) \|^n$.  If we choose $t_0,\mu_0>0$ small enough so that
\begin{equation*}
\|R \circ (B_{\epsilon} - \mu(\epsilon)) \| \leq r_0 \left( \mu_0 + \sum_{i=1}^{\infty} \epsilon^i c_0 \right) < 1,
\end{equation*}
we then have a convergent Neumann series.  This completes our proof.
\end{proof}

\begin{rem}
In our case of Laplace-Beltrami operators arising as perturbations from that on $S^2$, we can take $r_0 = 1$ as the spectral gaps $\delta$ can be taken to be as greater than or equal to $2$ always.
\end{rem}

\subsection{Proof of Theorem \ref{thm:rellich}: Weierstrass factorization and final steps}

Let us massage our series representation for the putative eigenfunction $\psi(\epsilon)$ in equation (\ref{eqn:putative_eigen}) a bit further, under the assumption that $\mu(\epsilon)$ is sufficiently small.  To provide some motivation, let $E_{\Lambda} = \mbox{span} \{ \phi_{\Lambda,m} \}_{m=1}^l$ and suppose that $\psi(\epsilon)$ is \textit{actually} an eigenvector $A_{\epsilon}$.  Then we know $\Pi_{E_{\Lambda}} \psi(\epsilon) = \sum_{m'=1}^l c_{m'}(\epsilon) \phi_{\Lambda,m'}$ for some values $c_{m'}(\epsilon)$; plugging in this linear combination into the series representation of $R(B_{\epsilon} - \mu)$ with the identity in Lemma \ref{lem:restriction_identity}, gives
\begin{equation} \label{eqn:first_solve}
\sum_{m'=1}^l c_{m'}(\epsilon) \langle \phi_{\Lambda, m},  (B_{\epsilon} - \mu(\epsilon)) \sum_{n=0}^{\infty} S^n \phi_{\Lambda,m'} \rangle = 0
\end{equation}
for each $m = 1, \dots, l$.

This set of equations (\ref{eqn:first_solve}) s commonly referred to as the ``first solvability conditions" in asymptotics.  Rellich's idea was to remove the dependence of $\mu$ on $\epsilon$ and treat it as an independent variable.  If we can solve these equations in $c_{m'}$ for sufficiently small $\mu$, then the theory of zeroes of analytic functions return our desired eigenvalues and eigenvectors.

We label $v_{m,m'}(\epsilon) := \langle \phi_{\Lambda, m}, (B_{\epsilon} - \mu) \sum_{n=0}^{\infty} S^n \phi_{\Lambda,m'} \rangle$; note that $v_{m,m'}(\epsilon) = \overline{v_{m',m}(\epsilon)}$ and hence the corresponding $l\times l$ matrix is Hermitian. We lift these functions into $\R^2_{\mu, \epsilon}$ and consider the resulting determinant of
\begin{equation} \label{eqn:spectral_det}
F(\mu, \epsilon) : = \det \left( v_{m,m'}(\epsilon)  \right)_{m,m'}
\end{equation}
For $|\epsilon| < \epsilon_0$ and $|\mu| < \mu_0$, we know that $F(\mu, \epsilon)$ is an analytic function in $\mu$ and $\epsilon$.  Notice that in the special case of $\mu=\widetilde{\mu}(\epsilon) = \Lambda - \Lambda(\epsilon)$, and if $|\mu(\epsilon)| < \mu_0$, then $F(\mu(\epsilon), \epsilon) = 0$ thus showing the existence of non-zero projection coefficients of $\overline{\psi}(t)$.

It is through the Weierstrass Preparation Theorem and that we can show the existence of such sufficiently small $\mu(\epsilon)$, and therefore $\Lambda(\epsilon)$, as encapsulated by the following proposition:
\begin{prop} \label{prop:weier_prep}
Consider the analytic function $F(\mu, \epsilon)$ defined in (\ref{eqn:spectral_det}).  We have that $F(\mu,0) = \mu^l$ and therefore
\begin{equation*}
F(\mu,\epsilon) = \left( \mu^l + p_{l-1}(\epsilon) \mu^{l-1} + \dots + p_1(\epsilon) \mu + p_0(\epsilon) \right) \times E(\mu,\epsilon)
\end{equation*}
where each $p_i(\epsilon)$ is analytic in $\epsilon$ and $E(\mu,\epsilon)$ is a non-vanishing analytic function in a rectangle $I(\mu_0)$.
\end{prop}

\begin{proof}
This is an immediate consequence of the Malgrange-Weierstrass Preparation Theorem with the derivative conditions applied in $\mu$.  For its rigorous statement, see the treatise by Guillemin-Golubitsky \cite[Chapter 4]{GG}.
\end{proof}

\begin{proof}[Proof that Proposition \ref{prop:weier_prep} implies Theorem \ref{thm:rellich}]
Take $\epsilon_0$ and $\mu_0$ small enough to satisfy the hypotheses of Proposition \ref{prop:weier_prep}.  Thanks to the corresponding matrix for the determinant $F$ being Hermitian, $F$ is itself real and we have a complete factorization of
\begin{equation*}
\left( \mu^l + p_{l-1}(\epsilon) \mu^{l-1} + \dots + p_1(\epsilon) \mu + p_0(\epsilon) \right)
\end{equation*}
into monomials of the form $\mu-\widetilde{\mu}_{m'}(\epsilon)$ for $m'=1,\dots,l$ where $|\mu-\widetilde{\mu}_{m'}(\epsilon)| \in [0, \mu_0)$ and $\epsilon\in [0,\epsilon_0)$.  By analyticity of $F(\mu, \epsilon)$, it follows that $\widetilde{\mu}_m(\epsilon)$ is analytic for all $m'$.

Notice that $F(\tilde{\mu}_m(\epsilon), \epsilon) = 0$, implying there exists non-trivial solution vectors $\overrightarrow{c}(\epsilon) \in \R^{l}$ where
\begin{equation} \label{eqn:solut_solvab}
\sum_{m'=1}^l c_{m'}(t) v_{m,m'}(\epsilon) = 0
\end{equation}
for all $m = 1, \dots, l$.  With this, we are ready to show $\psi(\epsilon)$ as defined in Lemma \ref{lem:neumann}, with $\mu = \tilde{\mu}_m(\epsilon)$ and $\overline{\psi}(\epsilon)$ having the coefficients $\overrightarrow{c}(\epsilon)$, is an eigenfunction of $A_{\epsilon}$ of eigenvalue $\Lambda_m(\epsilon) = \Lambda + \tilde{\mu}_m(\epsilon)$.

We have
\begin{align*}
 & (A_0 - \Lambda) \psi(\epsilon) \\
 & =  \underbrace{(A_0 - \Lambda)R(B_{\epsilon} - \tilde{\mu}_m(t))}_{\mbox{orthogonality}} \psi(\epsilon) \\
 & = \underbrace{(B_{\epsilon} - \tilde{\mu}_m(\epsilon)) \psi(\epsilon) - \Pi_{E_{\Lambda}} \left( (B_{\epsilon} - \tilde{\mu}_m(\epsilon)) \psi(\epsilon) \right)}_{\mbox{Lemma \ref{lem:pseudo_inv}}} \\
 & = \underbrace{(B_{\epsilon} - \tilde{\mu}_m(\epsilon)) \psi(\epsilon) - \sum_{m=1}^l  \left(  \sum_{m'=1}^l c_{m'}(\epsilon) \langle \phi_{\Lambda, m}, (B_{\epsilon} - \tilde{\mu}(\epsilon)) \sum_{n=0}^{\infty} S^n \phi_{\Lambda, m'} \rangle \right) \phi_{\Lambda, m}}_{\mbox{Lemma \ref{lem:neumann}}} \\
 & = (B_{\epsilon} - \tilde{\mu}_m(\epsilon)) \psi(\epsilon)
\end{align*}
with the last equation following from the implications of $F(\tilde{\mu}_m(\epsilon), \epsilon) = 0$ namely (\ref{eqn:solut_solvab}).  Therefore, $A_{\epsilon} \psi(\epsilon) = \left( \Lambda + \tilde{\mu}_m(\epsilon) \right) \psi(\epsilon)$, and finally yielding the eigenfunction we aimed for.  Since $A(\epsilon)$ is self-adjoint and $\Lambda$ is real, so is $\mu(\epsilon)$.  By positivity of $A_{\epsilon}$, we know that $\Lambda + \tilde{\mu}(\epsilon) \geq 0$ and $| \tilde{\mu}(\epsilon)| \leq \delta$.  Orthonormality follows from executing the Gram-Schmit procedure.

To prove the final part of the theorem, let $\Lambda' = \Lambda + (d_2 - d_1)$.  Then the operator norm of $\|\left( A_0 +  (2 \delta) \Pi_{E_{\Lambda}} \right) - \Lambda'\|_{op}$ is bounded below, as $ A_0 +  (2 \delta) \Pi_{E_{\Lambda}}$ has no spectrum in $J$, and this continues to be the case in the interval $[-d_1 + \Lambda, d_2 + \Lambda]$ for  parameters $|t| < t_1(\Lambda, \delta)$ where $\epsilon_1 \leq \epsilon_0$ thanks to analyticity which in turn provides continuity of our operator norms in $\epsilon$.  This completes our proof.
\end{proof}

\section{Biaxial Ellipsoids in $\mathbb{R}^3$} \label{sect:biax}

\subsection{Coordinate calcuations and reduced equations}

Let $\phi \in (0, \pi)$ and $\theta \in [0, 2\pi)$.  A natural set of ellipsoidal coordinates are $(a \sin \phi \cos \theta, a \sin \phi \cos \theta, b \cos
\phi)$.  We don't make any assumptions on the sizes of $a$ or $b$ yet.  The induced metric in these coordinates is of the form
\begin{align*}
g_{11} = a^2 \cos^2 \phi + b^2 \sin^2 \phi, g_{12}=g_{21}=0, g_{22} = a^2 \sin^2 \phi.
\end{align*}
Notice that the functions $g_{11}$ and $g_{22}$ are, respectively, the squared chord length for the ellipse in $\R^2$ given by $x^2/a^2 + z^2/b^2$ squared and the squared radius of the $S^1$ cross section of our ellipsoid $E_{a,b}$.

Therefore the corresponding Laplace-Beltrami operator takes the form
\begin{equation}
\Delta_{g_0} = (a^2 \cos^2 \phi + b^2 \sin^2 \phi)^{-1} \partial_{\phi \phi} + (a^2 \sin^2 \phi)^{-1} \partial_{\theta \theta} + h(\phi) \partial_{\phi},
\end{equation}
where $h(\phi) = \sqrt{\det(g)}^{-1} \partial_{\phi}(\sqrt{\det(g)} (a^2 \cos^2 \phi + b^2 \sin^2 \phi)^{-1})$.  In the upcoming sections, we set $a = 1 + \alpha \epsilon$ and $b = 1 + \beta \epsilon$ where $\alpha, \beta \neq 0$ independent in $\epsilon$.  For $\epsilon_0 = 1/2$, the metric $g_{\epsilon}$ for $E_{a,b}$ is analytic and admits a finite polynomial expansion in $\epsilon$.  Hence, Theorem \ref{thm:berger_lem} applies and we can perform calculations using analytic series in $\epsilon$ for a possibly smaller threshold $\epsilon_0$.

Thanks to the natural $S^1$ action on $E_{a,b}$, basic representation theory (see for instance Terras' treatise \cite{Ter}) tells us that $L^2(E_{a,b}, dV)$ has a basis consisting of separable eigenfunctions of the form $u(\phi) e^{im \theta}$ where $m \in \Z$.  Plugging this ansatz into the Laplace-Beltrami operator and performing the standard calculations leads to the following separated equations written in Sturm-Liouville form:
\begin{equation*}
\frac{ \left(  \frac{a \sin \phi}{\sqrt{a^2 \cos^2 \phi + b^2 \sin^2 \phi}} \partial_{\phi} \left(  \frac{a \sin \phi}{\sqrt{a^2 \cos^2 \phi + b^2 \sin^2 \phi}} \partial_{\phi} \ \right) + \Lambda a^2 \sin^2 \phi \right) u}{u} = m^2 = - \frac{\partial_{\theta \theta} f}{f}
\end{equation*}
where $f(\theta) = e^{im \theta}$.  Thus, the factor $u$ satisfies the following reduced equation
\begin{equation} \label{eqn:biaxial_ODE}
\left(  \frac{a \sin \phi}{\sqrt{a^2 \cos^2 \phi + b^2 \sin^2 \phi}} \partial_{\phi} \left(  \frac{a \sin \phi}{\sqrt{a^2 \cos^2 \phi + b^2 \sin^2 \phi}} \partial_{\phi} \ \right) + \Lambda a^2 \sin^2 \phi \right) u - m^2 u = 0,
\end{equation}
exhibiting some dependence on the integral parameter $m$; in the upcoming calculations, we incorporate the parameter $a^2$ into $\Lambda$ and abuse notation by calling the new eigenvalue $\Lambda(\epsilon)$.  We will use these equations to compute approximations for the purported basis of $L^2(E_{a,b}, dV)$ using the analytic perturbation introduced in Section \ref{sect:analyt_pert}.

\subsection{Deriving the first solvability condition}

The theory of 2nd-order ODEs tells us that the eigenvalues $\Lambda
(\epsilon )$, for each $m$, are simple. Thanks to the term $m^{2}$ and
therefore being able to use $\pm m$ to generate potentially different
eigenfunctions in $\phi $, we know that we can write each solution in $\phi $
as either $u_{m}^{+}$ or $u_{-m}$: it is of these $u_{\pm m}$ that we take
the expansion and in turn approximate $\Lambda ^{2}(\epsilon )$, which we
see has multiplicity at least $2$ for $m\neq 0$ and multiplicity at least 1
when $m=0$.

Theorem \ref{thm:rellich} holds for the Legendre equation on $[-1,1]$. Hence
we know that given $l(l+1)\in spec(S^{2})$, there exists $\epsilon
_{0}(S^{2},l(l+1))$ such that the following expansions are valid:
\begin{align*}
& \Lambda (\epsilon )=l(l+1)+\epsilon \Lambda _{1}+\epsilon ^{2}\Lambda
_{2}+\dots  \\
& u_{\pm m}(\phi ,\epsilon )=u_{0}(\phi )+\epsilon u_{1}(\phi )+\epsilon
^{2}u_{2}(\phi )+\dots.
\end{align*}%
We make note that we are abusing notation by using $\Lambda_1$ to represent the first-order coefficient of $\Lambda$; as we have fixed the eigenvalue $\Lambda$, this use is unambiguous.

The simplicity of the Sturm-Liouville spectrum is used when writing the
order 0 (in $\epsilon $) term for $u_{\pm m}$. To these expansions, we apply
\begin{equation*}
A_{\epsilon }=\frac{1}{\sin ^{2}\phi }\left( \frac{\sin \phi }{\sqrt{%
a^{2}\cos ^{2}\phi +b^{2}\sin ^{2}\phi }}\partial _{\phi }\left( \frac{\sin
\phi }{\sqrt{a^{2}\cos ^{2}\phi +b^{2}\sin ^{2}\phi }}\partial _{\phi
}\right) -\frac{m^{2}}{a^{2}\sin ^{2}\phi }\right)
\end{equation*}%
and work out the formal series for
\begin{equation*}
A_{\epsilon }u_{\pm m}(\epsilon )=-\Lambda (\epsilon )u_{\pm m}(\epsilon ).
\end{equation*}%
We expand $A_{\varepsilon }$ in $\varepsilon $ to obtain
\begin{equation*}
A_{\epsilon }=A_{0}+\varepsilon A_{1}+O(\varepsilon ^{2})
\end{equation*}%
where the big-O notation means that the \textquotedblleft implicit" object
is a 2nd-order differential operator, and with%
\begin{eqnarray}
A_{0} &=&\partial _{\phi \phi }+\cot \phi \partial _{\phi }-\frac{m^{2}}{%
\sin ^{2}\phi } \\
A_{1} &=&(\alpha -\beta )\left( 2\sin ^{2}\phi \partial _{\phi \phi }+2\sin
(2\phi )\partial _{\phi }\right) -2\alpha A_{0}
\end{eqnarray}%
A gathering of the 0th-order and 1st-order terms in $\epsilon $ yield the
following two equations:
\begin{align}
A_{0}u_{0}& =-l(l+1)u_{0},\mbox{ and }  \label{O1} \\
A_{0}u_{1}+l(l+1)u_{1}& =-A_{1}u_{0}-\Lambda _{1}u_{0}.  \label{O2}
\end{align}%
The solution to (\ref{O1})\ is given by the spherical harmonics $%
u_{0}=c_{0}P_{l}^{m}(\cos \phi )$ where $c_{0}$ is any constant.

The operator $A_{0}$ is self-adjoint with respect to the measure $\sin \phi
d\phi $, whose corresponding inner product we denote by $\langle ,\rangle $.
Taking the product of both sides of (\ref{O2})\ with respect to $u_{0}$ we
then obtain: $\left\langle u_{0},A_{0}u_{1}+l(l+1)u_{1}\right\rangle
=\left\langle u_{1},A_{0}u_{0}+l(l+1)u_{0}\right\rangle =0,$ so that $%
-\langle u_{0},A_{1}u_{0}\rangle -\Lambda _{1}\langle u_{0},u_{0}\rangle =0$
or
\begin{equation}
\Lambda _{1}=\frac{-\langle u_{0},A_{1}u_{0}\rangle }{\langle
u_{0},u_{0}\rangle }.  \label{Lam1}
\end{equation}

Next we compute,
\begin{equation*}
A_{1}u_{0}=\left( \alpha -\beta \right) \left\{ \left( \left( -l\left(
l+1\right) 2\sin ^{2}\phi +2m^{2}\right) u_{0}+2\sin \phi \cos \phi \partial
_{\phi }u_{0}\right) \right\} +2\alpha l\left( l+1\right) u_{0}
\end{equation*}%
where we used the Legendre equation (\ref{O1})\ to rewrite $\partial _{\phi
\phi }u_{0}=\left( -\cot \phi \partial _{\phi }+\frac{m^{2}}{\sin ^{2}\phi }%
-l(l+1)\right) u_{0}.$ Therefore (\ref{Lam1}) yields the formula

\begin{equation*}
\Lambda _{1}=\left( \beta -\alpha \right) \frac{\int \left\{ \left(
-2l\left( l+1\right) \sin ^{2}\phi +2m^{2}\right) u_{0}+2\sin \left( \phi
\right) \cos \phi \partial _{\phi }u_{0}\right\} u_{0}\sin \phi d\phi }{%
\int_{0}^{\pi }u_{0}^{2}\sin \phi d\phi .}-2\alpha l\left( l+1\right) .
\end{equation*}%
Changing variables $\cos \phi =t,$ we then obtain%
\begin{equation}
\Lambda _{1}=\left( \beta -\alpha \right) \frac{\int_{-1}^{1}\left\{ \left(
-2l\left( l+1\right) \left( 1-t^{2}\right) +2m^{2}\right) P^{2}(t)-2\left(
1-t^{2}\right) tP^{\prime }(t)P(t)\right\} dt}{\int_{0}^{\pi }P(t)^{2}dt.}%
-2\alpha l\left( l+1\right)  \label{2:24}
\end{equation}%
where $P(t)=\frac{\partial ^{l+m}}{dt^{(l+m)}}\left[ \left( 1-t^{l}\right)
^{m}\right] .$ Integrating by parts, we rewrite:%
\begin{equation*}
\int_{-1}^{1}2\left( 1-t^{2}\right) tP^{\prime }P=-\int_{-1}^{1}\left(
1-3t^{2}\right) P^{2}
\end{equation*}%
so that (\ref{2:24})\ becomes%
\begin{equation*}
\Lambda _{1}=\left( \beta -\alpha \right) \left\{ \left( -2l\left(
l+1\right) +2m^{2}+1\right) +(2l\left( l+1\right) -3)\frac{%
\int_{-1}^{1}t^{2}P^{2}dt}{\int_{0}^{\pi }P^{2}dt.}\right\} -2\alpha l\left(
l+1\right)
\end{equation*}%
Finally, the ratio $\frac{\int_{-1}^{1}t^{2}P^{2}dt}{\int_{-1}^{1}P^{2}}=%
\frac{2l^{2}-2m^{2}+2l-1}{\left( 2l+3\right) \left( 2l-1\right) }$ is
evaluated in Appendix \ref{app}. This yields formula (\ref{eqn:Lambda1biaxial}).

The multiplicity statement in Theorem \ref{thm:biaxial} follows easily from here
after taking into account that for each $l(l+1)$, we get a single analytic
eigenvalue for $m=0$ and get a double analytic eigenvalue for each $m=\pm
1,\dots ,\pm l$. Now apply Corollary \ref{cor:all_analyt} which implies that
the multiplicities cannot be higher than 2.

\section{Triaxial ellipsoids in $\mathbb{R}^{3}$}  \label{sect:triax}

\subsection{Coordinate calculations}
We now pursue the calculation of eigenvalues for triaxial ellipsoids. For
organizational reasons, let us write $\Delta_{g}$ in coordinates. In
$\mathbb{R}^{3}$, we take the coordinates $(a\cos\phi\cos\theta,b\cos\phi
\sin\theta,c\sin\phi)$ where $\phi\in(0,\pi)$ and $\theta\in\lbrack0,2\pi)$.
This leads us the following expression of $\Delta_{g}$ as follows:
\[
\Delta_{g}=A\partial_{\phi\phi}+B\partial_{\phi\theta}+C\partial_{\theta
\theta}+E\partial_{\phi}+F\partial_{\theta}%
\]
where%
\begin{align*}
A  &  :=\frac{g_{22}}{D},\ \ B:=-2\frac{g_{12}}{D},\ \ C:=\frac{g_{11}}{D},\\
E  &  :=D^{-1/2}\left(  \left(  g_{22}D^{-1/2}\right)  _{\phi}-\left(
g_{12}D^{-1/2}\right)  _{\theta}\right) \\
&  =\frac{\partial_{\phi}g_{22}}{D}-\frac{1}{2}\frac{g_{22}}{D^{2}}%
\partial_{\phi}D-\frac{\partial_{\theta}g_{12}}{D}+\frac{1}{2}\frac{g_{12}%
}{D^{2}}\partial_{\theta}D\\
F  &  :=D^{-1/2}\left(  \left(  g_{11}D^{-1/2}\right)  _{\theta}-\left(
g_{12}D^{-1/2}\right)  _{\phi}\right) \\
&  =\frac{\partial_{\theta}g_{11}}{D}-\frac{1}{2}\frac{g_{11}}{D^{2}}%
\partial_{\theta}D-\frac{\partial_{\phi}g_{12}}{D}+\frac{1}{2}\frac{g_{12}%
}{D^{2}}\partial_{\phi}D
\end{align*}
where%
\begin{align*}
D(\theta,\phi)  &  :=g_{11}g_{22}-g_{12}g_{21}\\
g_{11}(\theta,\phi)  &  :=\cos^{2}\phi\left(  a^{2}\cos^{2}\theta+b^{2}%
\sin^{2}\theta\right)  +c^{2}\sin^{2}\phi\\
g_{22}(\theta,\phi)  &  :=\sin^{2}\phi\left(  a^{2}\sin^{2}\theta+b^{2}%
\cos^{2}\theta\right) \\
g_{12}(\theta,\phi)  &  :=g_{21}=\left(  b^{2}-a^{2}\right)  \frac{\sin\left(
2\phi\right)  \sin\left(  2\theta\right)  }{4}%
\end{align*}

\subsection{Deriving the first solvability condition}
We want our triaxial ellipsoid to be a small perturbation of $S^{2}$, so we
set
\[
a:=1+\alpha\varepsilon,\ \ \ b:=1+\beta\varepsilon,\ \ c:=1+\gamma\varepsilon
\]
Note that our metric coefficients $g_{ij}$ are analytic and non-vanishing in
$\varepsilon$, therefore making the inverse metric $g^{ij}$'s coefficients
analytic as well.  We are in a position to apply Theorem \ref{thm:berger_lem}.

Thus, we can carefully massage $\Delta_{g}$ into an analytic
series of operators, specifically as
\[
\Delta_{g}=A_{0}+\varepsilon A_{1}+\varepsilon^{2}A_{2}(\varepsilon)
\]
where%
\[
A_{0}=\partial_{\phi\phi}+\frac{1}{\sin^{2}\phi}\partial_{\theta\theta}%
+\frac{\cos\phi}{\sin\phi}\partial_{\phi}%
\]
and%
\begin{align*}
A_{1}  &  :=\left[  \left(  2\beta-2\alpha\right)  \cos^{2}\theta\cos^{2}%
\phi+\left(  2\gamma-2\beta\right)  \cos^{2}\phi-2\gamma\right]
\partial_{\phi\phi}+{\frac{\left(  \left(  2\alpha-2\beta\right)  \cos
^{2}\theta-2\alpha\right)  }{\sin^{2}\phi}}\partial_{\theta\theta}\\
&  +4\left(  \alpha-\beta\right)  {\frac{\cos\phi\sin\theta\cos\theta}%
{\sin\phi}}\partial_{\phi\theta}+4\left(  \beta-\alpha\right)  {\frac
{\sin\theta\cos\theta}{\sin^{2}\phi}}\partial_{\theta}\\
&  +4{\left[  \left(  \left(  \beta-\alpha\right)  \left(  \cos\left(
\theta\right)  \right)  ^{2}-\beta+\gamma\right)  \cos^{2}\phi+\frac{3}%
{2}\left(  \alpha-\beta\right)  \cos^{2}\theta-\frac{\alpha}{2}+\beta
-\gamma\right]  }\frac{\cos\phi}{\sin\phi}\partial_{\phi}.
\end{align*}

Now, we set $u$ to be a solution of%
\[
\Delta_{g}u=-\Lambda(\epsilon)u.
\]
Expanding similarly as in Section \ref{sect:biax} thanks to Theorem
\ref{thm:berger_lem}, we are can again write the following two analytic series
in $\varepsilon:$
\begin{align*}
u  &  =u_{0}+\varepsilon u_{1}+\ldots\\
\Lambda(\varepsilon)  &  =l(l+1)+\varepsilon\Lambda_{1}+\ldots
\end{align*}
Gathering terms in $\varepsilon$ yields the 0th-order equation of
\[
A_{0}u_{0}=-l(l+1)u_{0}%
\]
for $l\in\mathbb{Z}^{+}$ and whose solution is given by
\begin{align*}
u_{0}  &  =C_{0}v_{0}(\theta,\phi)+\sum_{m=1}^{l}C_{m}v_{m}(\theta,\phi
)+D_{m}w_{m}(\theta,\phi),\ \ \ \text{where}\\
v_{m}  &  =\cos(m\theta)P_{l}^{m}(\cos(\phi));\ \ w_{m}=\sin(m\theta)P_{l}%
^{m}(\cos(\phi));
\end{align*}
the set $\{w_{0},\{w_{m},v_{m}\}_{m=1}^{l}\}$ form a basis for the space of
spherical harmonics with eigenvalue $l(l+1)$ with the $P_{l}^{m}$ being the
associated Legendre functions given by
\[
P_{l}^{m}(t)=A_{l,m}\left(  1-t^{2}\right)  ^{m/2}\frac{\partial^{m+l}%
}{\partial t^{m+l}}\left[  \left(  1-t^{2}\right)  ^{l}\right]  .
\]
The normalization constants $A_{l,m}$ are chosen so that%
\begin{equation}
\left\langle v_{m},v_{m}\right\rangle =\left\langle w_{m},w_{m}\right\rangle
=1, \label{v1}%
\end{equation}
where the inner product $\left\langle ,\right\rangle $ is with respect to the
metric on $S^{2}$, that is
\[
\left\langle v,w\right\rangle :=\int_{0}^{2\pi}\int_{0}^{\pi}v(\theta
,\phi)w(\theta,\phi)\sin\phi d\phi d\theta.
\]

At order $\epsilon$ we have
\begin{equation}
A_{0}u_{1}+l(l+1)u_{1}=-\Lambda_{1}u_{0}-A_{1}u_{0}.
\label{eqn:triax_1st_solv}%
\end{equation}
Note that $A_{0}$ is self-adjoint with respect to the inner product
$\langle,\rangle$. Multiplying (\ref{eqn:triax_1st_solv}) by $v_{k}$ or
$w_{k}$ and integrating we then obtain the solvability conditions%
\begin{align*}
\left\langle v_{k},u_{0}\right\rangle \Lambda_{1}  &  =-\left\langle
v_{k},A_{1}u_{0}\right\rangle ,\ \ \ k=0\ldots l,\\
\left\langle w_{k},u_{0}\right\rangle \Lambda_{1}  &  =-\left\langle
w_{k},A_{1}u_{0}\right\rangle ,\ \ \ k=1\ldots l.
\end{align*}
Thanks to orthonormality, we have that $\left\langle v_{k}u_{0}\right\rangle
=C_{k}$ and $\left\langle w_{k}u_{0}\right\rangle =D_{k}.$

Thus, in contrast to the biaxial case, the first solvability equation returns
a more complicated system of linear equations to solve. This eigenvalue
problem problem for $\Lambda_{1}$ can be read off as
\begin{equation}
\label{eqn:eigen_prob_triax}M \overline{V} = \Lambda_{1} \overline{V}%
\end{equation}
where $M$ is the $\left(  2l+1\right)  \times\left(  2l+1\right)  $ matrix and
$\overline{V}$ is $2l+1$ column vector containing coefficients $C_{m}$ and
$D_{m}.$ We now seek to simplify the ``matrix elements" $\left\langle
v_{m},A_{1}u_{0}\right\rangle .$

\subsection{Expanding into Fourier modes}

To execute this simplification of the quantities $\left\langle
v_{m},A_{1}u_{0}\right\rangle$ , we expand in terms of Fourier modes on $S^{1}$ in the variable
$\theta$. For simplicity of notation, set $P(\phi):=P_{l}^{m}(\cos(\phi))$. We
find that\bes\label{A1u}
\begin{equation}
A_{1}\left(  \cos(m\theta)P\right)  =g_{m-}(\phi)\cos\left(  \left(
m-2\right)  \theta\right)  +g_{m}(\phi)\cos\left(  m\theta\right)
+g_{m+}(\phi)\cos\left(  \left(  m+2\right)  \theta\right)
\end{equation}
and%
\begin{equation}
A_{1}\left(  \sin(m\theta)P\right)  =g_{m-}(\phi)\sin\left(  \left(
m-2\right)  \theta\right)  +g_{m}(\phi)\sin\left(  m\theta\right)
+g_{m+}(\phi)\sin\left(  \left(  m+2\right)  \theta\right)
\end{equation}
\ees and\bes\label{g}%

\begin{align}
g_{m-}(\phi)  &  =\frac{\beta-\alpha}{2}\left(  l\left(  l+1\right)  \sin
^{2}\phi+\left(  -l(l+1)-m^{2}\right)  +\frac{2m(m-1)}{\sin^{2}\phi}\right)
P\nonumber\\
&  +\frac{\beta-\alpha}{2}\left(  {-\cos\phi\sin\phi+}\frac{2(m-1)\cos\phi
}{\sin\phi}\right)  P_{\phi},\\
g_{m+}(\phi)  &  =\frac{\beta-\alpha}{2}\left(  l\left(  l+1\right)  \sin
^{2}\phi+\left(  -l(l+1)-m^{2}\right)  +\frac{2m(m+1)}{\sin^{2}\phi}\right)
P\nonumber\\
&  +\frac{\beta-\alpha}{2}\left(  {-\cos\phi\sin\phi+}\frac{2(-m-1)\cos\phi
}{\sin\phi}\right)  P_{\phi},\\
g_{m}(\phi)  &  =\left(  \left(  \alpha+\beta-2\gamma\right)  {m}^{2}+l\left(
l+1\right)  2\gamma+\left(  \alpha+\beta-2\gamma\right)  l(l+1)\cos^{2}%
\phi\right)  P+\nonumber\\
&  \left(  \alpha+\beta-2\gamma\right)  \sin\left(  \phi\right)  \cos\left(
\phi\right)  P_{\phi}.
\end{align}
\ees In above expressions, we have used the standard Legendre ODE (albeit with
$t=\cos\phi$), $P_{\phi\phi}=-\frac{\cos\phi}{\sin\phi}P_{\phi}+\left[
\frac{m^{2}}{\sin^{2}\phi}-l\left(  l+1\right)  \right]  P$ in order to
eliminate any occurrence of $P_{\phi\phi}.$

In total, each quantity $\left\langle v_{m},A_{1}u_{0}\right\rangle $
simplifies to
\begin{align*}
\left\langle v_{m},A_{1}u_{0}\right\rangle  &  =\left\langle v_{m},A_{1}%
v_{m}\right\rangle C_{m}+\left\langle v_{m},A_{1}v_{m+2}\right\rangle
C_{m+2}+\left\langle v_{m},A_{1}v_{m-2}\right\rangle C_{m-2}\\
\left\langle w_{m},A_{1}u_{0}\right\rangle  &  =\left\langle w_{m},A_{1}%
w_{m}\right\rangle D_{m}+\left\langle w_{m},A_{1}w_{m+2}\right\rangle
D_{m+2}+\left\langle w_{m},A_{1}w_{m-2}\right\rangle D_{m-2}%
\end{align*}
where we used the convention that $C_{m}=0$ if $m\notin\left\{  0,1,\ldots
,l\right\}  $ and similarly $D_{m}=0$ if $m\notin\left\{  1,2,\ldots
,l\right\}  $. \ In other words, odd-frequency Fourier modes couple with
neighbouring odd-frequency modes, and even-frequency Fourier modes couple with
their neighbors. Therefore the eigenvalue problem (\ref{eqn:eigen_prob_triax})
decomposes into four distinct subproblems:\ one for even-indexed $C$'s, one
for odd-indexed $C$'s, and so forth. As a result, $\Lambda_{1}$ is
characterized by the following.

\begin{prop}
\label{prop:triax_entries}Define\bes%
\begin{equation*}
l_{e}=\left\{
\begin{array}
[c]{c}%
l,\text{ if }l\text{ is even}\\
l-1,\text{ if }l\text{ is odd}%
\end{array}
\right.  ;\ \ \ \ \ \ l_{o}=\left\{
\begin{array}
[c]{c}%
l,\text{ if }l\text{ is odd}\\
l-1,\text{ if }l\text{ is even}%
\end{array}
\right.
\end{equation*}
and define the tridiagonal matrices\label{M}{\small
\begin{equation}
M_{\cos,e}=-\left[
\begin{array}
[c]{ccccc}%
\left\langle v_{0},A_{1}v_{0}\right\rangle  & \left\langle v_{0},A_{1}%
v_{2}\right\rangle  & 0 & \ldots & 0\\
\left\langle v_{2},A_{1}v_{0}\right\rangle  & \left\langle v_{2},A_{1}%
v_{2}\right\rangle  & \left\langle v_{2},A_{1}v_{4}\right\rangle  & \ddots &
\vdots\\
0 & \left\langle v_{4},A_{1}v_{2}\right\rangle  & \left\langle v_{4}%
,A_{1}v_{4}\right\rangle  & \ddots & 0\\
\vdots & \ddots & \ddots & \ddots & \left\langle v_{l_{e}-2},A_{1}v_{l_{e}%
}\right\rangle \\
0 & \ldots & 0 & \left\langle v_{l_{e}},A_{1}v_{l_{e}-2}\right\rangle  &
\left\langle v_{l_{e}},A_{1}v_{l_{e}}\right\rangle
\end{array}
\right]
\end{equation}%
\begin{equation}
M_{\cos,o}=-\left[
\begin{array}
[c]{ccccc}%
\left\langle v_{1},A_{1}v_{1}\right\rangle  & \left\langle v_{1},A_{1}%
v_{3}\right\rangle  & 0 & \ldots & 0\\
\left\langle v_{3},A_{1}v_{1}\right\rangle  & \left\langle v_{3},A_{1}%
v_{3}\right\rangle  & \left\langle v_{3},A_{1}v_{5}\right\rangle  & \ddots &
\vdots\\
0 & \left\langle v_{5},A_{1}v_{3}\right\rangle  & \left\langle v_{7}%
,A_{1}v_{7}\right\rangle  & \ddots & 0\\
\vdots & \ddots & \ddots & \ddots & \left\langle v_{l_{o}-2},A_{1}v_{l_{0}%
}\right\rangle \\
0 & \ldots & 0 & \left\langle v_{l_{0}},A_{1}v_{l_{0}-2}\right\rangle  &
\left\langle v_{l_{0}},A_{1}v_{l_{0}}\right\rangle
\end{array}
\right]
\end{equation}%
\begin{equation}
M_{\sin,e}=-\left[
\begin{array}
[c]{ccccc}%
\left\langle w_{2},A_{1}w_{2}\right\rangle  & \left\langle w_{2},A_{1}%
w_{4}\right\rangle  & 0 & \ldots & 0\\
\left\langle w_{4},A_{1}w_{2}\right\rangle  & \left\langle w_{4},A_{1}%
w_{4}\right\rangle  & \left\langle w_{4},A_{1}w_{6}\right\rangle  & \ddots &
\vdots\\
0 & \left\langle w_{6},A_{1}w_{4}\right\rangle  & \left\langle w_{6}%
,A_{1}w_{8}\right\rangle  & \ddots & 0\\
\vdots & \ddots & \ddots & \ddots & \left\langle w_{l_{e}-2},A_{1}w_{l_{e}%
}\right\rangle \\
0 & \ldots & 0 & \left\langle w_{l_{e}},A_{1}w_{l_{e}-2}\right\rangle  &
\left\langle w_{l_{e}},A_{1}w_{l_{e}}\right\rangle
\end{array}
\right]
\end{equation}%
\begin{equation}
M_{\sin,o}=-\left[
\begin{array}
[c]{ccccc}%
\left\langle w_{1},A_{1}w_{1}\right\rangle  & \left\langle w_{1},A_{1}%
w_{3}\right\rangle  & 0 & \ldots & 0\\
\left\langle w_{3},A_{1}w_{1}\right\rangle  & \left\langle w_{3},A_{1}%
w_{3}\right\rangle  & \left\langle w_{3},A_{1}w_{5}\right\rangle  & \ddots &
\vdots\\
0 & \left\langle w_{5},A_{1}w_{3}\right\rangle  & \left\langle w_{7}%
,A_{1}w_{7}\right\rangle  & \ddots & 0\\
\vdots & \ddots & \ddots & \ddots & \left\langle w_{l_{o}-2},A_{1}w_{l_{o}%
}\right\rangle \\
0 & \ldots & 0 & \left\langle w_{l_{0}},A_{1}w_{l_{0}-2}\right\rangle  &
\left\langle w_{l_{o}},A_{1}w_{l_{o}}\right\rangle
\end{array}
\right]
\end{equation}
} \ees

Then for $\Lambda=l\left(  l+1\right)  +\varepsilon\Lambda_{1}+O\left(  \varepsilon
^{2}\right),$ as given by Theorem \ref{thm:berger_lem}, we have that $\Lambda_{1}$ is one of the $2l+1$ eigenvalues of the four matrices
(\ref{M}). 
\end{prop}

The entries of the matrices (\ref{M}) are computable explicitly -- see
Appendix \ref{app} -- with the following result:
\bes\label{entries}
\begin{align}
\left\langle v_{m},A_{1}v_{m}\right\rangle  &  =2\gamma l(l+1)+(\alpha
+\beta-2\gamma)\frac{2l\left(  l+1\right)  }{\left(  2l+3\right)  \left(
2l-1\right)  }\left(  l^{2}+m^{2}+l-1\right)  \mbox{ for }m\neq1,\\
\left\langle w_{m},A_{1}w_{m}\right\rangle  &  =\left\langle v_{m},A_{1}%
v_{m}\right\rangle ,\mbox{ for }m\neq1,\\
\left\langle v_{1},A_{1}v_{1}\right\rangle  &  =(\frac{3\alpha}{2}+\frac
{\beta}{2}-2\gamma)\frac{2l^{2}\left(  l+1\right)  ^{2}}{\left(  2l+3\right)
\left(  2l-1\right)  }+2l\left(  l+1\right)  \gamma,\\
\left\langle w_{1},A_{1}w_{1}\right\rangle  &  =(\frac{3\beta}{2}+\frac
{\alpha}{2}-2\gamma)\frac{2l^{2}\left(  l+1\right)  ^{2}}{\left(  2l+3\right)
\left(  2l-1\right)  }+2l\left(  l+1\right)  \gamma,\\
\left\langle v_{m-2},A_{1}v_{m}\right\rangle  &  =\left\langle v_{m}%
,A_{1}v_{m-2}\right\rangle =\left(  \beta-\alpha\right)  \frac{l\left(
l+1\right)  }{\left(  2l-1\right)  \left(  2l+3\right)  }\nonumber\\
\times &  \sqrt{\left(  l-m+1\right)  (l-m+2)\left(  l+m-1\right)
(l+m)}\times\left\{
\begin{array}
[c]{c}%
1\text{ if }m>2\text{ }\\
\sqrt{2}\text{ if }m=2,
\end{array}
\right. \\
\mbox{ and } \left\langle w_{m-2},A_{1}w_{m}\right\rangle  &  =\left\langle w_{m}%
,A_{1}w_{m-2}\right\rangle =\left\langle v_{m-2},A_{1}v_{m}\right\rangle
,\ \text{for}\ m\geq3.
\end{align}
\ees

\subsection{Explicit calculations for $l=1,2,3$}

We carry out Proposition \ref{prop:triax_entries} in the simple cases of $%
l=1,2$ and observe the form of the corresponding perturbed eigenvalues.

For $l=1$, we arrive at three possible choices for $\Lambda_{1}$: 
\begin{align}
\Lambda_{1} & =-\frac{4}{5}\left( \alpha+\beta+3\gamma\right) \text{ with }%
u_{0}=v_{0};  \label{12:09} \\
\Lambda_{1} & =-\frac{4}{5}\left( 3\alpha+\beta+\gamma\right) \text{ with }%
u_{0}=v_{1};  \label{12:10} \\
\Lambda_{1} & =-\frac{4}{5}\left( \alpha+3\beta+\gamma\right) \text{ with }%
u_{0}=w_{1}.   \label{12:11}
\end{align}
Note that this coincides with the formula (\ref{eqn:Lambda1biaxial})\ for
the biaxial case by taking $\gamma=\alpha:$ the formula (\ref{12:11})
corresponds to mode $m=0$ while (\ref{12:09})\ and (\ref{12:10})\ correspond
to the mode $m=\pm1.$ For a generic choice where all of $\alpha,\beta,$ $%
\gamma$ are distinct, we find that spectrum of $\Delta_{g}$ near $1\cdot2=2$
is simple and is of the form $2+\epsilon\Lambda_{1}+\mathcal{O}(\epsilon^{2})
$.

For $l=2,$ matrices (\ref{M})\ are of the size 2x2 1x1, 1x1 and 1x1,
respectively. The five resulting eigenvalues, respectively, are: 
\begin{align}
\Lambda _{1}& =-4\left( \alpha +\beta +\gamma \right) \pm \frac{16}{7}\sqrt{%
\alpha ^{2}+\beta ^{2}+\gamma ^{2}-\alpha \beta -\alpha \gamma -\beta \gamma 
}, \notag
\\
& -\frac{12}{7}\left( 3\alpha +\beta +3\gamma \right) -\frac{12}{7}\left(
3\alpha +3\beta +\gamma \right) ,\ -\frac{12}{7}\left( \alpha +3\beta
+3\gamma \right) .  
 \label{2:11}
\end{align}%
Again, for a generic choice where all of $\alpha ,\beta ,$ $\gamma $ are
distinct, we find that spectrum of $\Delta _{g}$ near to $2\cdot 3=6$ is
simple and of the form $6+\epsilon \Lambda _{1}+\mathcal{O}(\epsilon ^{2})$.
When $\gamma =\alpha $, eigenvalues (\ref{2:11})\ become, in the order as
listed,%
\begin{equation*}
\Lambda _{1}=-\frac{40}{7}\alpha -\frac{44}{7}\beta ,\ \ -\frac{72}{7}\alpha
-\frac{12}{7}\beta ,\ -\frac{72}{7}\alpha -\frac{12}{7}\beta ,\ -\frac{48}{7}%
\alpha -\frac{36}{7}\beta ,\ -\frac{48}{7}\alpha -\frac{36}{7}\beta ,\ 
\end{equation*}%
and as expected, they coincide with the formula (\ref{eqn:Lambda1biaxial}),
with $m=0,\ \pm 2,\ \pm 2,\ \pm 1,\ \pm 1$, respectively.

Finally, for $l=3,$ matrices (\ref{M})\ are of the size 2x2, 2x2, 2x2 and
1x1, respectively, and yield the following eigenvalues for the correction $%
\Lambda _{1}:$%
\begin{eqnarray*}
M_{\cos ,e} &:&-\frac{104\alpha }{15}-\frac{104\beta }{15}-\frac{152\gamma }{%
15}\pm \frac{32}{15}\sqrt{4\alpha ^{2}+4\beta ^{2}-7\alpha \beta -\alpha
\gamma -\beta \gamma +\gamma ^{2}} \\
M_{\cos ,o} &:&-\frac{104\gamma }{15}-\frac{104\beta }{15}-\frac{152\alpha }{%
15}\pm \frac{32}{15}\sqrt{4\gamma ^{2}+4\beta ^{2}-7\gamma \beta -\gamma
\alpha -\beta \alpha +\alpha ^{2}} \\
M_{\sin ,e} &:&-8\alpha -8\beta -8\gamma  \\
M_{\sin ,o} &:&-\frac{104\gamma }{15}-\frac{104\alpha }{15}-\frac{152\beta }{%
15}\pm \frac{32}{15}\sqrt{4\gamma ^{2}+4\alpha ^{2}-7\gamma \alpha -\gamma
\beta -\alpha \beta +\beta ^{2}}.
\end{eqnarray*}%
As with previous cases, we have verified that all these eigenvalues are
distinct whenever $\alpha ,\beta ,\gamma $ are distinct. Taking $\gamma
=\beta $, they become%
\begin{eqnarray*}
M_{\cos ,e} &:&-\frac{64}{5}\alpha -\frac{56}{5}\beta ,\ \ \ -\frac{64\alpha 
}{5}-\frac{8}{5}\beta  \\
M_{\cos ,o} &:&-\frac{64}{5}\alpha -\frac{56}{5}\beta ,\ \ \ \ -\frac{%
64\alpha }{5}-\frac{8}{5}\beta  \\
M_{\sin ,e} &:&-16a-8b\\
M_{\sin ,o} &:&-16a-8b,\ \ \ \ -\frac{176}{5}a-\frac{184}{5}b .
\end{eqnarray*}%
These agree with formula (\ref{eqn:Lambda1biaxial}), with $m=\pm 1,\ \ \pm
3,\ \pm 1,\ \pm 3,\ \pm 2,\  \pm 2,\ 0,$ respectively.

\section{Numerical experiments and conjectures}

\label{sect:numerics}

\subsection{Biaxial eigenvalues}

\label{sect:biaxial_numerics}

Equation (\ref{eqn:biaxial_ODE}) can be used to compute the eigenvalues
numerically for any bi-axial ellipsoid. To solve (\ref{eqn:biaxial_ODE})
numerically, we discretize the space using the usual centered difference
discretization. The procedure leads to an $N\times N$-matrix eigenvalue
problem $Av=\Lambda v$ where the eigenfunction $v(\phi)$ is approximated by $%
v(\phi_{j})\approx v_{j},\ \ \phi_{j}=\frac{j\pi}{N}.$ We impose Neumann
boundary conditions at the poles: $v^{\prime}(0)=0,\ \ v^{\prime}(\pi)=0.$
In our computations, we have set $N=400$ which is sufficient to compute $%
\Lambda_{\text{numeric}}$ to 4 significant digits; we verified that doubling
the mesh-size did not change the answer within that precision.

Table \ref{table:1} provides a comparison between the analytical result for $%
\Lambda_{1}$, as given in equation (\ref{eqn:Lambda1biaxial}), and the
numerical approximation using the above procedure. The axial parameters $a=1$
and $b=1+\varepsilon$ are considered with $\varepsilon=0.1$ and $\varepsilon
=0.05$. While we did not prove this analytically, the numerics suggest that
the relative error scales linearly with $\varepsilon$, as would be expected
assuming that $\Lambda$ is analytic in $\varepsilon.$

\begin{table}[ptb]
\begin{equation*}
\begin{tabular}{|l|l|l||l|l|l||l|l|l|}
\hline
\multicolumn{3}{|l||}{} & \multicolumn{3}{l||}{$\varepsilon=0.1$} & 
\multicolumn{3}{l|}{$\varepsilon=0.05$} \\ \hline
$m$ & $\Lambda_{0}$ & $\Lambda_{1}$ & $\Lambda_{\text{numeric}}$ & $\frac{%
\Lambda_{\text{numeric}}-\Lambda_{0}}{\varepsilon}$ & \%err & $\Lambda_{%
\text{numeric}}$ & $\frac{\Lambda_{\text{numeric}}-\Lambda_{0}}{\varepsilon}$
& \%err \\ \hline
0 & 2 & -2.400 & 1.7772 & -2.228 & 7.17\% & 1.8844 & -2.312 & 3.67\% \\ 
\hline
& 6 & -6.2857 & 5.4079 & -5.921 & 5.80\% & 5.6950 & -6.100 & 2.95\% \\ \hline
& 12 & -12.266 & 10.8463 & -11.53 & 6.00\% & 11.4052 & -11.896 & 3.02\% \\ 
\hline
1 & 2 & -0.8 & 1.9250 & -0.750 & 6.25\% & 1.9612 & -0.776 & 3.00\% \\ \hline
& 6 & -5.1429 & 5.5227 & -4.773 & 7.19\% & 5.7521 & -4.958 & 3.60\% \\ \hline
& 12 & -11.2 & 10.9509 & -10.49 & 6.34\% & 11.4573 & -10.854 & 3.09\% \\ 
\hline
2 & 6 & 1.7143 & 5.8404 & -1.596 & 6.90\% & 5.9173 & -1.654 & 3.52\% \\ 
\hline
& 12 & 8 & 11.2595 & -7.405 & 7.44\% & 11.6152 & -7.696 & 3.80\% \\ \hline
3 & 12 & -2.666 & 11.7527 & -2.473 & 7.24\% & 11.8716 & -2.568 & 3.68\% \\ 
\hline
\end{tabular}
\end{equation*}
\caption{Comparison between asymptotic value of $\Lambda_{1}$ and the
estimated numerical value for the biaxial ellipsoid. Here, $a=1$ and $b=1+%
\protect\varepsilon$ (corresponding to $\protect\alpha=0,\protect\beta=1$)
and with $\protect\varepsilon=0.1$ or 0.05. The error column is the relative
error between $\Lambda_{1}$ and $\frac{\Lambda_{\text{numeric}}-\Lambda_{0}}{%
\protect\varepsilon}.$}
\label{table:1}
\end{table}

\subsection{Triaxial eigenvalues}

We used the algorithm described in \cite{Mac} to compute the eigenvalues
numerically for a true triaxial ellipsoid. As opposed to the numerical
method used for the biaxial eigenvalues, the code used in this case
implements the closest-point algorithm developed in \cite{cpm1, cpm2}. To
compare with the numerics, we take $a=1,b=1+\varepsilon,c=1-\varepsilon,$ so
that $\left( \alpha,\beta,\gamma\right) =\left( 0,1,-1\right) .$ Table \ref%
{table:2} compares the numerics with the analytic formulas for first nine
triaxial eigenvalue subprincipal terms, as obtained in Proposition \ref%
{prop:triax_entries}.

\begin{table}[ptb]
\begin{equation*}
\begin{tabular}{|l|l||l|l|l||l|l|l|}
\hline
\multicolumn{2}{|l||}{} & \multicolumn{3}{l||}{$\varepsilon=0.2$} & 
\multicolumn{3}{l|}{$\varepsilon=0.1$} \\ \hline
$\Lambda_{0}$ & $\Lambda_{1}$ & $\Lambda_{\text{numeric}}$ & $\frac {%
\Lambda_{\text{numeric}}-\Lambda_{0}}{\varepsilon}$ & {err} & $\Lambda _{%
\text{numeric}}$ & $\frac{\Lambda_{\text{numeric}}-\Lambda_{0}}{\varepsilon }
$ & err \\ \hline
{2} & {-1.6} & 1.69763 & -1.511 & {0.088} & 1.84107 & -1.589 & 0.0106 \\ 
\hline
{2} & {0} & 2.05566 & 0.278 & {0.278} & 2.01048 & 0.1047 & 0.105 \\ \hline
{2} & {1.6} & 2.33333 & 1.666 & {0.066} & 2.1603 & 1.603 & 0.003 \\ \hline
{6} & {-3.95897} & 5.24037 & -3.798 & {0.160} & 5.59748 & -4.025 & -0.066 \\ 
\hline
{6} & {-3.42857} & 5.45296 & -2.735 & {0.693} & 5.68282 & -3.171 & 0.256 \\ 
\hline
{6} & {0} & 6.04863 & 0.2431 & {0.243} & 6.0048 & 0.0479 & 0.0479 \\ \hline
{6} & {3.42857} & 6.82912 & 4.145 & {0.717} & 6.36925 & 3.692 & 0.263 \\ 
\hline
{6} & {3.95897} & 6.83013 & 4.150 & {0.191} & 6.3857 & 3.857 & -0.101 \\ 
\hline
\end{tabular}
\ 
\end{equation*}
\caption{Comparison between asymptotic value of $\Lambda_{1}$ and the
estimated numerical value for the triaxial ellipsoid. Here, $(a,b,c)=(1+%
\protect\varepsilon,b=1-\protect\varepsilon,1)$ (corresponding to $(\protect%
\alpha =0,\protect\beta,\protect\gamma)=(1,-1,0)$) and with $\protect%
\varepsilon=0.2$ or 0.1. The error column is the difference between $%
\Lambda_{1}$ and $\frac{\Lambda _{\text{numeric}}-\Lambda_{0}}{\protect%
\varepsilon}.$}
\label{table:2}
\end{table}

The code in \cite{Mac} generates a sparse matrix that corresponds to a
discretization of the Laplace-Beltrami operator for a surface. We use the
spatial resolution of $0.1$ that results in a $30,000\times30,000$ sparse
matrix: the lowest eigenvalue is zero with the next eight eigenvalues being
listed in the second columns of our given tables.

We note that even with a sparse matrix of dimension $3 \times 10^4$, the
control on the error is rather poor; we expect no more than 2-3 digits of
precision. For this reason, the error does not scale linearly in $\varepsilon
$ as would be expected:\ it would require too many meshpoints to resolve up
to $\mathcal{O}(\varepsilon^{2})$ numerically for this two-dimensional
non-symmetric problem:\ the method relies on the MATLAB sparse eigenvalue
solver \texttt{eigs} which is not sufficiently accurate for such large
matrices. Indeed, this problem provides a good test case for the
closest-point method algorithm.

{We conclude this numerical exploration with a discussion of
multiplicity of eigenvalues. For the biaxial case, due to the monotonicity
of the formula (\ref{eqn:Lambda1biaxial}) with respect to $m$, there are
exactly $l$ double eigenvalues (corresponding to $m=1\ldots l$) and a one
single eigenvalue (corresponding to $m=0$) near $\Lambda_{0}=l\left(
l+1\right) $, for a total of $2l+1$ eigenvalues. For a triaxial $%
\alpha\neq\beta\neq\gamma$, extensive numerical experiments indicate that
for a fixed $l,$ the perturbations $\Lambda_{1}$ as given in Proposition \ref%
{prop:triax_entries} are all distinct (we verified this analytically when $%
l\leq3$ since $\Lambda_{1}$ is an eigenvalue of at most 2x2 matrix in that
case, and numerically for $l$ up to 5). As a consequence, we repeat the
following generally held belief in the spectral-geometry literature: Suppose 
$\alpha\neq\beta\neq\gamma.$ Given any $L$, there exists $\varepsilon_{0}$
such that for all $\varepsilon<\varepsilon_{0},$ the set $\{ \Lambda : \,
\Lambda \leq L \}$ contains only simple eigenvalues.

\subsection{Observations on nodal domains}

Figure \ref{fig:efun} shows the first few eigenmodes with $l=1,2,3,4,$ as
given by Proposition \ref{prop:triax_entries}, for several values of $\alpha
,\beta,\gamma.$ The corresponding eigenfunctions are plotted, as well as
their nodal lines.

A \textit{nodal domain} of $u$ is a connected component of the subset $%
\Omega(u)=\{x\in M:u(x)\neq0\}$, the regions of $M$ where $u$ is positive or
negative. In this section, we briefly explore the nodal domain structures
(or better yet, shapes) of our ellipsoidal harmonics. Recall that a version
of Courant's nodal domain theorem says that on a compact, boundaryless
manifold, the number of nodal domains for eigenfunctions produced from the $n
$-th eigenspace (counting multiplicities) is bounded above by $n$.

In the biaxial case near a sphere, there are $2l+1$ eigenvalues near the
eigenvalue $l(l+1)$ with $l$ double eigenvalues and one simple eigenvalue.
The corresonding eigenfunctions have $P_{l}^{m}(\phi)\cos(m\theta)$, $%
m=0\ldots l$ and $P_{l}^{m}(\phi)\sin(m\theta)$, $m=1\ldots l$ as their
leading order terms. These are shown in Figure \ref{fig:efun} (rows 1 to 4)
for $l=1\ldots4.$ Note that those corresponding to double eigenvalues, with
leading order terms $P_{l}^{m}(\phi)(C\cos(m\theta)+D\sin(m\theta)$, all
have $\left( l+1-m\right) 2m$ nodal domains when $0<m\leq l;$ whereas the
simple eigenvalue $P_{l}^{0}(\phi)$ has $\left( l+1\right) $ nodal domains.
Moreover, formula (\ref{eqn:Lambda1biaxial})\ shows that the $2l+1$
eigenvalues near $l\left( l+1\right) $ are monotone in $m.$ This allows for
a full characterization for the number of nodal domains in the biaxial case.

Figure \ref{fig:efun} suggests that when deforming a biaxial ellipsoid to a
triaxial ellipsoid, the nodal line topology changes only at the
\textquotedblleft north\textquotedblright\ and \textquotedblleft
south\textquotedblright\ poles, where $2m$ nodal lines intersect, and does
so in two very specific ways. For example, when $m=3,$ the pole is
desingularized either in this way: \includegraphics[width=0.15%
\textheight]{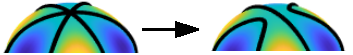} or this way:\ \includegraphics[width=0.15%
\textheight]{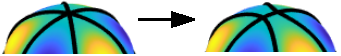}. The latter transformation does not affect the
number of nodal domains, whereas the former reduces it by either $2\left(
m-1\right) $ if $m<l$, or by $m-1$, if $m=l.$ Based on extensive numerical
observations for $l=0\ldots5$, we offer the following conjecture on the
number of nodal domains for our near-sphere ellipsoids.
\begin{figure}[ptb]
\includegraphics[width=0.95\textwidth]{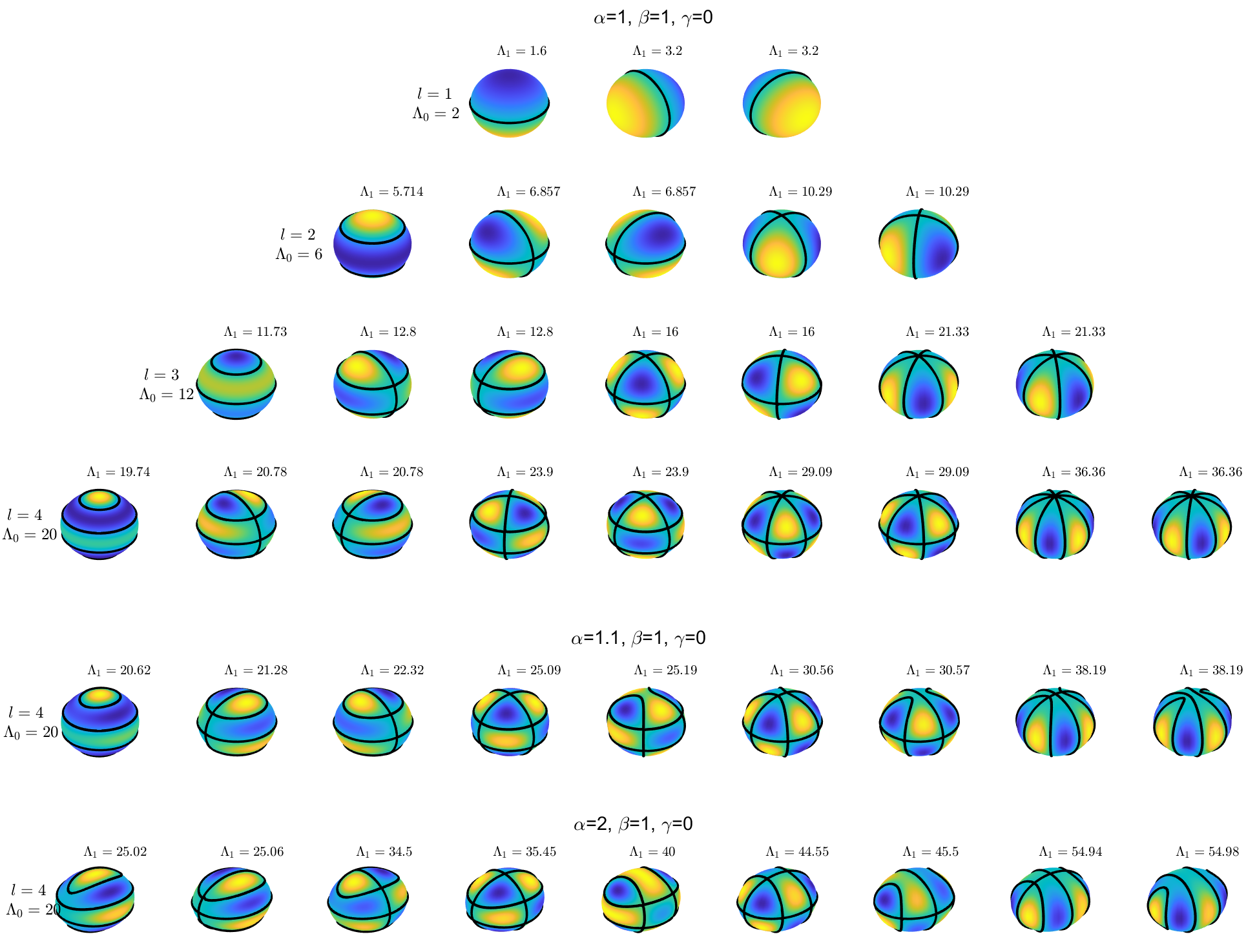}
\caption{The first few eigenfunctions of some ellipsoids, computed from
Proposition \protect\ref{prop:triax_entries}. Corresponding nodal lines and $%
\Lambda_{1}$ are also shown. Parameters are as indicated. }
\label{fig:efun}
\end{figure}

\begin{figure}[ptb]
\includegraphics[width=0.98\textwidth]{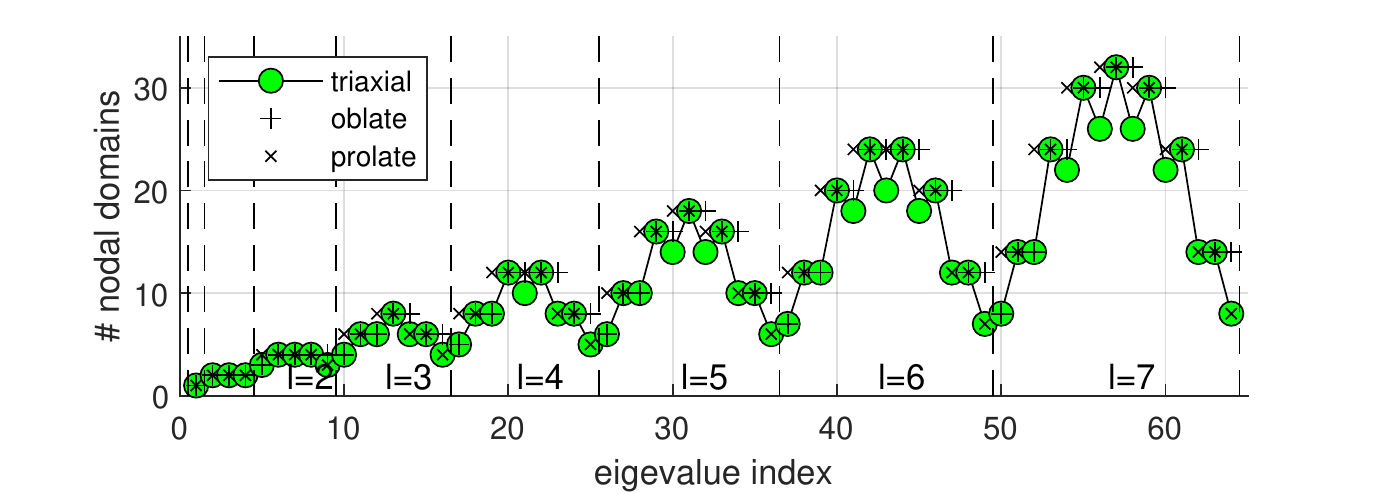}
\caption{The number of nodal domains versus eigenvalue index for the three
cases as given in Conjecture \protect\ref{conj:nodal}. Notice the
fluctuations against Courant's bound.}
\label{fig:nodal}
\end{figure}

\begin{conj}
\label{conj:nodal}Define the sequence $N_{k},$ $k=0\ldots2l$ as follows:%
\begin{equation}
N_{0}=l+1;\ N_{2m-1}=(l+1-m)2m,\ \ N_{2m}=N_{2m-1},\ m=1\ldots l; 
\label{142}
\end{equation}
and sequence $\hat{N}_{k},$ $k=0\ldots2l$ as follows: 
\begin{equation}
\hat{N}_{2m}=\left\{ 
\begin{array}{c}
N_{2m}-2\left( m-1\right) ,\ \ m=1\ldots l-1 \\ 
l+1,\ \ m=l%
\end{array}
\right. ;\ \ \ \ \hat{N}_{2m-1}=N_{2m-1}.
\end{equation}
 For $\epsilon_0$ sufficiently small such that Theorem \ref%
{thm:triaxial} holds, arrange the $2l+1$ eigenvalues near the the level $%
\Lambda_{0}=l(l+1)$ in increasing order. Their nodal domain count is as follows:

\begin{itemize}
\item[(a)] For an oblate ellipsoid ($\alpha=\gamma>\beta$),
the nodal domain count is $N_{0},\ldots,N_{2l}.$

\item[(b)] For a prolate ellipsoid ($\alpha=\gamma<\beta$),
the nodal domain count is $N_{2l},\ldots,N_{0}.$

\item[(c)] For a triaxial ellipsoid ($\alpha\neq\beta\neq\gamma)$, the nodal
domain count is $\hat{N}_{0},\hat{N}_{1},\ldots\hat{N}_{2l}.$\bigskip
\end{itemize}
\end{conj}

For example take $l=4.$ Then the three sequences in Conjecture \ref%
{conj:nodal} are:

\begin{itemize}
\item[(a)] Oblate:$\ N_{0},\ldots,N_{2l}\ \ =$ $5,\ 8,\ \ \ 8,\ 12,\ 12,\
12,\ 12,\ 8,\ 8.$

\item[(b)] Prolate:$\ N_{2l},\ldots,N_{0}\ =$ $8,\ 8,\ 12,\ 12,\ 12,\ 12,\ \
\ 8,\ 8,\ 5.$

\item[(c)] Triaxial:$\ \hat{N}_{0},\ldots,\hat{N}_{2l}=5,\ 8,\ \ \ 8,\ 12,\
10,\ 12,\ \ \ 8,\ 8,\ 5.$
\end{itemize}

Part of the motivation for this conjecture comes from the
observed low multiplicities in the spectrum of our ellipsoids. Conjecture %
\ref{conj:nodal} states (and the reader can verify)\ that the sequence (a)
corresponds to the number of nodal domains in Figure \ref{fig:efun} (row 4)\
whereas the sequence (c)\ describes the number of nodal domains for rows 5
and 6. Note that \emph{any} triaxial ellipsoid near the sphere is
conjectured to have the same nodal sequence, regardless of the relative
sizes of $\alpha,\beta,\gamma.$ This is reflected in the fact\ the sequence $%
\hat{N}_{k}$ is symmetric: $\hat{N}_{k}=\hat{N}_{2l-k}.$ We verified this
conjecture numerically, for numerous values of $\alpha,\beta,\gamma$ and
with $l$ up to 5.

In Figure \ref{fig:nodal} we plot the number of nodal domains as a function
of the eigenvalue index for an arbitrary triaxial ellipsoid, with parameter $%
l$ up to $7.$ For large $l,$ the maximum number of nodal domains asymptotes
to $\sim l^{2}$ (by taking $m=l/2$ in (\ref{142})) whereas the lowest
asymptotes to $l$ (corresponding to $m=0$).

It is worth noting the works of Levy and Eremenko-Jakobson-Nadirashvilli on
nodal domains on $S^2$. Levy \cite{Lev} constructs high-frequency examples
of spherical harmonics that obtain exactly two nodal domains.
Eremenko-Jakobson-Nadirashvilli \cite{EJN} construct harmonics that obtain
various prescribed topological configurations in their nodal structure. Both
works use perturbation-type arguments.

\subsection{Numerics for large perturbations of spheres}

\begin{figure}[ptb]
\label{fig:p_to_c} \includegraphics[width=0.95%
\textwidth]{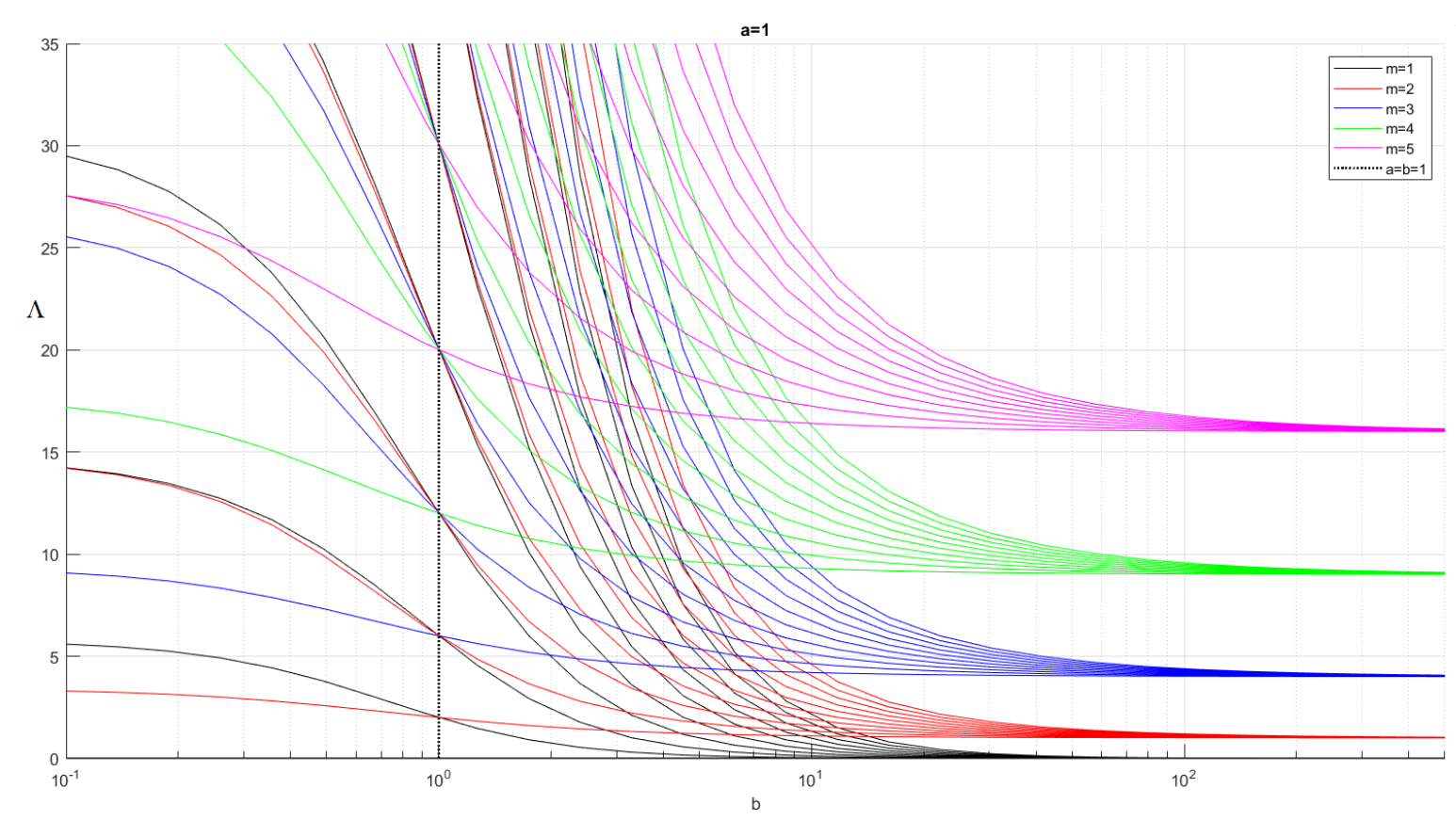}
\caption{$\Lambda$ graphed as a function of the axial parameter $b$ with $a=1
$ fixed, for the bi-axial ellipsoid ($c=a$). The first ten eigenvalues for
modes $m=1,\ldots,5$ are plotted.}
\label{fig:biaxial}
\end{figure}

\begin{figure}[ptb]
\includegraphics[width=0.5 \textwidth]{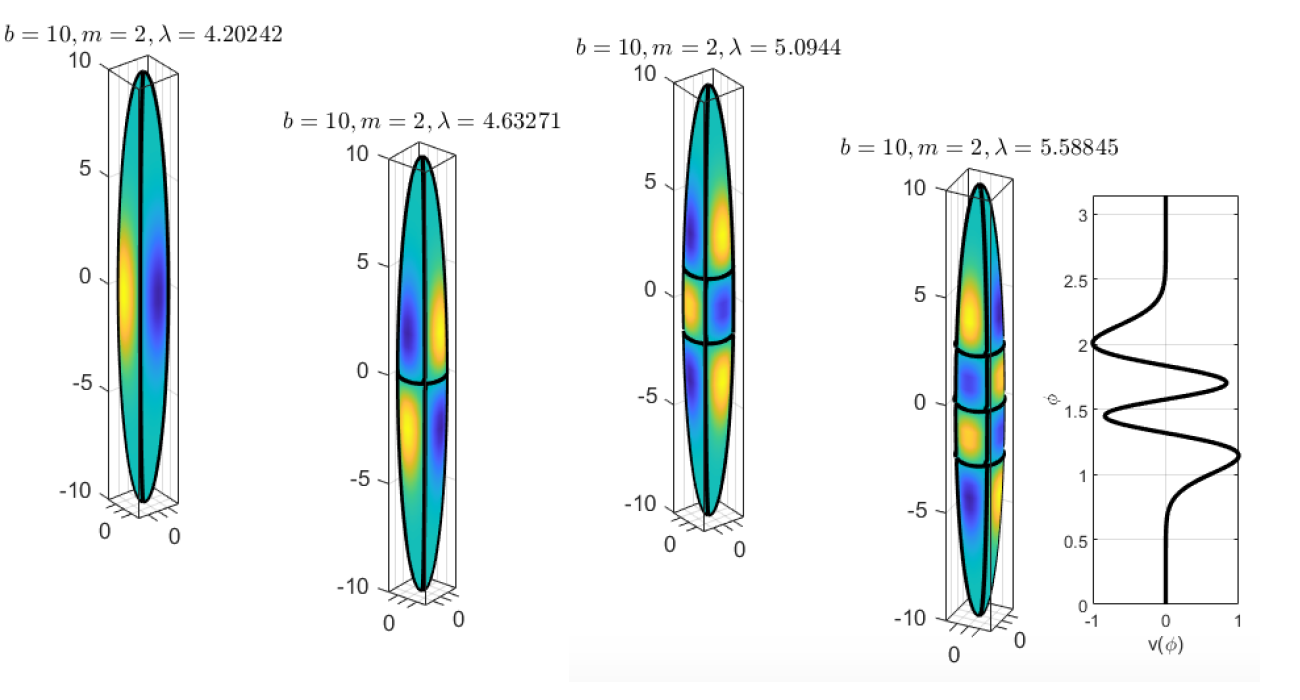}
\caption{Eigenfunctions on cigar-type ellipsoids; in the insert gives the
axial profile of an eigenfunction with eigenvalue $\Lambda=5.58845$. The
solid black lines indicate the zero curves of the plotted eigenfunctions.}
\label{fig:cigar}
\end{figure}

\begin{figure}[ptb]
\includegraphics[width=0.9 \textwidth]{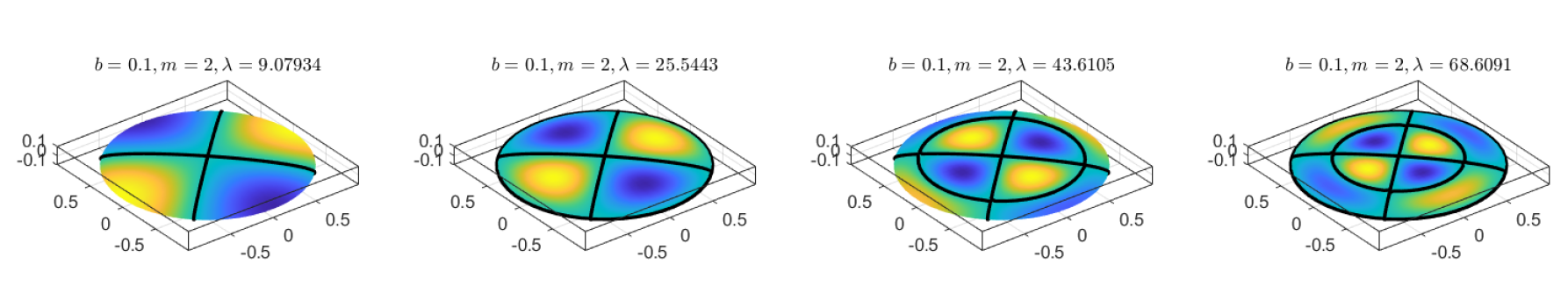}
\caption{Eigenfunctions on disk-type ellipsoids. The solid black lines
indicate the zero curves of the plotted eigenfunctions.}
\label{fig:pancake}
\end{figure}

Figure \ref{fig:biaxial} shows the numerically computed eigenvalues of a
biaxial ellipsoid using the method described in Section \ref%
{sect:biaxial_numerics}. In this regime, we set $a=1$ and varied $b$ from $%
0.1$ to $500.$

The case of the sphere corresponds to the solid black vertical line $b=1$,
and as expected, multiple eigenvalues collide at this point with $%
\Lambda=l(l+1).$ There are numerous eigenvalue crossings far away from the
sphere.

With $a=1$ and in the limit $b\rightarrow\infty,$ the ellipsoid takes a
cigar-type shape. In this case, \emph{all} of the eigenfunctions
corresponding to mode $m$ appear to asymptote to $\Lambda\sim m^{2}.$
Furthermore, the eigenfunctions appear to be \textquotedblleft
microlocalized" (that is, exhibiting its main oscillations) near the center
of the cigar, as illustrated in Figure \ref{fig:cigar}. It is an interesting
open question to explain this \textquotedblleft
microlocalization\textquotedblright\ in this asymptotic regime.

With $a=1$ and in the limit $b\rightarrow0,$ the ellipsoid degenerates into
a two-dimensional disk. In this case, the problem appears to degenerate into
a union of the eigenvalues of a unit disk with either Dirichlet or Neumann
boundary conditions as illustrated in Figure \ref{fig:pancake}. In this
limit, the numerics suggest that the eigenvalues approach roots of either $%
J_{m}(\sqrt{\Lambda})=0$ or $J_{m}^{\prime}(\sqrt{\Lambda})=0\,$where $J_{m}$
is the Bessel function of order $m.$ This behaviour is reminiscent of
eigenvalue asymptotics in the presence of degenerating metrics at least in
the case of the \textquotedblleft singular" manifold being boundaryless; see
for instance \cite{S13} and the references therein.

\section{Appendix A: Calculation of (\protect\ref{entries})}

\label{app}

We will use the notation%
\begin{align}
v_{m}(\phi,\theta) & =\cos\left( m\theta\right) P_{l}^{m}(\cos \phi),\ \ \ \
w_{m}(\phi,\theta)=\sin\left( m\theta\right) P_{l}^{m}(\cos\phi),  \notag \\
P_{l}^{m}(t) & :=A_{l,m}Q_{l}^{m}(t),  \notag \\
Q_{l}^{m}\left( t\right) & :=\left( 1-t^{2}\right) ^{m/2}\frac {%
\partial^{m+l}}{\partial t^{m+l}}\left[ \left( 1-t^{2}\right) ^{l}\right] \\
\text{where} & \text{ }A_{l,m}\text{ is chosen so that }\int_{0}^{2\pi}%
\int_{0}^{\pi}v_{m}(\phi,\theta)\sin\phi d\phi d\theta=1.  \notag
\end{align}
A couple of key integrals that appear in the computations are:%
\begin{equation}
J_{0}=\int_{-1}^{1}\left( 1-t^{2}\right) ^{l}dt;\ \ \ \
J_{2}=\int_{-1}^{1}t^{2}\left( 1-t^{2}\right) ^{l}dt.  \label{J02}
\end{equation}
All of the quantities will be ultimately expressed in terms of their ratio:%
\begin{equation}
\frac{J_{2}}{J_{0}}=\frac{1}{2l+3}.  \label{Jratio}
\end{equation}

We start with the following lemma.

\begin{lem}
\label{lemma:ratio}%
\begin{equation}
\frac{\int_{-1}^{1}t^{2}\left( P_{l}^{m}(t)\right) ^{2}dt}{%
\int_{-1}^{1}\left( P_{l}^{m}(t)\right) ^{2}dt}=\frac{\int_{-1}^{1}t^{2}%
\left( Q_{l}^{m}(t)\right) ^{2}dt}{\int_{-1}^{1}\left( Q_{l}^{m}(t)\right)
^{2}dt}=\frac{2l^{2}-2m^{2}+2l-1}{\left( 2l+3\right) \left( 2l-1\right) }.
\label{418}
\end{equation}
\end{lem}

\textbf{Proof. }This follows by successive integration by parts. We start
with $\int_{-1}^{1}\left( Q_{l}^{m}(t)\right) ^{2}:$

\begin{align}
\int_{-1}^{1}\left( Q_{l}^{m}(t)\right) ^{2} & =\int_{-1}^{1}\left(
1-t^{2}\right) ^{m}\frac{\partial^{m+l}}{\partial t^{m+l}}\left[ \left(
1-t^{2}\right) ^{l}\right] \frac{\partial^{m+l}}{\partial t^{m+l}}\left[
\left( 1-t^{2}\right) ^{l}\right] dt  \notag \\
& =\int_{-1}^{1}\left( \frac{\left( 2l\right) !}{(l-m)!}t^{m+l}+\ldots%
\right) \frac{\partial^{m+l}}{\partial t^{m+l}}\left[ \left( 1-t^{2}\right)
^{l}\right] dt \\
& =\frac{\left( 2l\right) !\left( l+m\right) !}{(l-m)!}J_{0}
\end{align}
where $J_{0}$ is as in (\ref{J02}). A similar computation yields%
\begin{align}
\int_{-1}^{1}\left( Q_{l}^{m}(t)\right) ^{2}t^{2}dt & =\int_{-1}^{1}\left(
t^{2}\left( 1-t^{2}\right) ^{m}\frac{\partial^{m+l}}{\partial t^{m+l}}\left[
\left( 1-t^{2}\right) ^{l}\right] \right) \frac {\partial^{m+l}}{\partial
t^{m+l}}\left[ \left( 1-t^{2}\right) ^{l}\right] dt  \notag \\
& =\frac{\left( 2l\right) !\left( m+l\right) !}{(l-m)!}\left( \frac{%
2l^{2}-2m^{2}+2l-1}{\left( 2l+3\right) \left( 2l-1\right) }\right) J_{0},
\end{align}
where (\ref{Jratio}) was used.

Equation (\ref{418})\ follows immediately. $\blacksquare$

\textbf{Derivation of (\ref{entries}a-d).} From (\ref{A1u}) we obtain%
\begin{equation}
\left\langle v_{m},A_{1}v_{m}\right\rangle =\frac{1}{\int_{0}^{\pi }\left(
P_{l}^{m}(\cos \phi )\right) ^{2}\sin \phi d\phi }\int_{0}^{\pi }d\phi {%
\times }\left\{ 
\begin{array}{c}
g_{m}(\phi )P_{l}^{m}(\cos \phi )\sin \phi ,\ m\neq 1 \\ 
g_{1}(\phi )P_{l}^{1}(\cos \phi )\sin \phi +g_{1-}(\phi )P_{l}^{1}(\cos \phi
)\sin \phi ,\ m=1%
\end{array}%
\right.  \label{236a}
\end{equation}%
and similarly, 
\begin{equation}
\left\langle w_{m},A_{1}w_{m}\right\rangle =\frac{1}{\int_{0}^{\pi }\left(
P_{l}^{m}(\cos \phi )\right) ^{2}\sin \phi d\phi }\int_{0}^{\pi }d\phi {%
\times }\left\{ 
\begin{array}{c}
g_{m}(\phi )P_{l}^{m}(\cos \phi )\sin \phi ,\ m\neq 1 \\ 
g_{1}(\phi )P_{l}^{1}(\cos \phi )\sin \phi -g_{1-}(\phi )P_{l}^{1}(\cos \phi
)\sin \phi ,\ \ m=1%
\end{array}%
\right.  \label{236b}
\end{equation}%
Rewrite (\ref{g}c) as%
\begin{equation*}
g_{m}(\phi )=\left\{ A+B\cos ^{2}\phi \right\} P_{l}^{m}(\cos \phi )+C\sin
\left( \phi \right) \cos \left( \phi \right) {\partial }_{\phi
}P_{l}^{m}(\cos \phi )
\end{equation*}%
where 
\begin{equation*}
C=\left( \alpha +\beta -2\,\gamma \right) ,\ \ A=C{m}^{2}+2\gamma l\left(
l+1\right) ,\ \ \ \ B=Cl(l+1).
\end{equation*}%
Then

\begin{align*}
\int_{0}^{\pi }g_{m}(\phi )P_{l}^{m}(\cos \phi )\sin \phi d\phi &
=\int_{-1}^{1}dt\left\{ \left( P_{l}^{m}(t)\right) ^{2}\left(
A+Bt^{2}\right) -Ct\left( 1-t^{2}\right) P_{l}^{m}(t)\partial
_{t}P_{l}^{m}(t)\right\} \\
& =\left( A+\frac{C}{2}\right) \int_{-1}^{1}\left( P_{l}^{m}(t)\right)
^{2}dt+\left( B-\frac{3C}{2}\right) \int_{-1}^{1}\left( P_{l}^{m}(t)\right)
^{2}t^{2}dt
\end{align*}%
This yields%
\begin{align}
\frac{\int_{0}^{\pi }g_{m}(\phi )P_{l}^{m}(\cos \phi )\sin \phi d\phi }{%
\int_{0}^{\pi }\left( P_{l}^{m}(\cos \phi )\right) ^{2}\sin \phi d\phi }& =A+%
\frac{C}{2}+\left( B-\frac{3C}{2}\right) \frac{\int_{-1}^{1}\left(
P_{l}^{m}(t)\right) ^{2}t^{2}dt}{\int_{-1}^{1}\left( P_{l}^{m}(t)\right)
^{2}dt}  \notag \\
& =(\alpha +\beta -2\gamma )\frac{2l\left( l+1\right) }{\left( 2l+3\right)
\left( 2l-1\right) }\left( l^{2}+m^{2}+l-1\right) +2l\left( l+1\right)
\gamma .  \label{236}
\end{align}%
where we used Lemma \ref{lemma:ratio}.

In the case when $m=1,$we also need to evaluate $\frac{\int_{0}^{\pi
}g_{1-}(\phi )P_{l}^{1}(\cos \phi )\sin \phi d\phi }{\int_{0}^{\pi }\left(
P_{l}^{1}(\cos \phi )\right) ^{2}\sin \phi d\phi }.\ $We compute%
\begin{align*}
\frac{2}{\alpha -\beta }\int_{0}^{\pi }g_{1-}(\phi )P_{l}^{1}(\cos \phi
)\sin \phi d\phi & =\int_{-1}^{1}\left( 1+l\left( l+1\right) t^{2}\right)
\left( P_{l}^{1}(t)\right) ^{2}-t\left( 1-t^{2}\right) \partial _{t}\left(
P_{l}^{1}(t)\right) P_{l}^{1}(t)dt \\
& =\frac{3}{2}\int_{-1}^{1}\left( P_{l}^{1}(t)\right) ^{2}dt+\left( l\left(
l+1\right) -\frac{3}{2}\right) \int_{-1}^{1}t^{2}\left( P_{l}^{1}(t)\right)
^{2}dt
\end{align*}%
so that, using Lemma \ref{lemma:ratio} we obtain 
\begin{equation}
\frac{\int_{0}^{\pi }g_{1-}(\phi )P_{l}^{1}(\cos \phi )\sin \phi d\phi }{%
\int_{0}^{\pi }\left( P_{l}^{1}(\cos \phi )\right) ^{2}\sin \phi d\phi }%
=\left( \alpha -\beta \right) \frac{l^{2}\left( l+1\right) ^{2}}{\left(
2l+3\right) \left( 2l-1\right) }.  \label{237}
\end{equation}%
Combining (\ref{236a}, \ref{236b}, \ref{236}, \ref{237})\ yields (\ref%
{entries}a-d).

\bigskip

\bigskip

\textbf{Derivation of (\ref{entries}e-f)}

Note that $\left\langle v_{m-2},L_{1}v_{m}\right\rangle
=C_{m-2}\int_{0}^{\pi }g_{m-}(\phi )P_{l}^{m-2}(\cos \left( \phi )\right)
\sin \phi d\phi ,\ \ $where $C_{m}=\int_{0}^{2\pi }\cos ^{2}\left( \theta
m\right) d\theta =\left\{ 
\begin{array}{c}
2\pi ,\ \ m=0 \\ 
\pi ,\ \ m\geq 1%
\end{array}%
\right. .$ We further write%
\begin{align*}
\int_{0}^{\pi }& g_{m-}(\phi )P_{l}^{m-2}(\cos \phi )\sin \phi d\phi =\frac{%
\beta -\alpha }{2}A_{l,m}A_{l,m-2} \\
& \times \left\{ \left( l+1\right) lI_{13}+\left( -l(l+1)-m^{2}\right)
I_{12}+2m(m-1)I_{11}+I_{2}\right\}
\end{align*}

where%
\begin{align*}
I_{11} & =\int_{0}^{\pi}\frac{Q_{l}^{m}(\cos\phi)Q_{l}^{m-2}(\cos\phi )}{%
\,\sin\phi}d\phi=0;\ \  \\
I_{12} & =\int_{0}^{\pi}Q_{l}^{m}(\cos\phi)Q_{l}^{m-2}(\cos\phi)\sin\phi
d\phi;\ \ \  \\
I_{13} & =\int_{0}^{\pi}Q_{l}^{m}(\cos\phi)Q_{l}^{m-2}(\cos\phi)\sin^{3}\phi
d\phi; \\
I_{2} & =\int_{0}^{\pi}\cos\left( \phi\right) \left( {-\sin}^{2}{\phi +}%
2(m-1)\right) \partial_{\phi}Q_{l}^{m}(\cos\phi)Q_{l}^{m-2}(\cos\phi )d\phi
\\
\frac{1}{C_{m}A_{m,l}^{2}} & =\int_{0}^{\pi}\left( Q_{l}^{m}(\cos
\phi)\right) ^{2}\sin\phi d\phi
\end{align*}

All these integrals are all evaluated using successive integration parts,
until they are expressed in terms $\int_{-1}^{1}\left( 1-t^{2}\right) ^{l}dt$
and $\int_{-1}^{1}t^{2}\left( 1-t^{2}\right) ^{l}dt.$ Skipping the details,
we obtain

\begin{align*}
I_{11}& =\int \frac{P_{l}^{m}P_{l}^{m-2}}{\,\sin \phi }d\phi =0, \\
I_{12}& =-\frac{(2l)!(m-2+l)!}{(l-m)!}J_{0} \\
I_{13}& =\frac{\left( 2l\right) !\left( m+l\right) !}{\left( l-m\right) !}%
\left( \frac{1}{2}J_{2}-\left( m+\frac{\left( l-m\right) \left( l-m-1\right) 
}{2\left( 2l-1\right) }\right) \frac{1}{\left( m+l\right) \left(
m+l-1\right) }J_{0}\right) \\
\frac{1}{C_{m}A_{m,l}^{2}}& =\frac{\left( 2l\right) !\left( l+m\right) !}{%
(l-m)!}\frac{\left( 2l\right) !\left( l+m\right) !}{(l-m)!}J_{0}
\end{align*}%
After some algebra we then obtain

\begin{equation*}
\left\langle v_{m-2},L_{1}v_{m}\right\rangle =\left( \beta -\alpha \right) 
\sqrt{\frac{C_{m-2}}{C_{m}}}\frac{l\left( l+1\right) }{\left( 2l-1\right)
\left( 2l+3\right) }\sqrt{(l-m+1)(l-m+2)\left( l+m-1\right) \left(
l-m\right) }
\end{equation*}%
which shows\ (\ref{entries}e). Formula (\ref{entries}f) is evaluated
analogously.

\section{Acknowledgements}

We are grateful to Tony Wong for his help with codes from \cite{Mac} that were instrumental in generating numerics for the triaxial ellipsoid.  Thanks also to Nilima Nigam, Iosif Polterovich, and Holger Dullin for their insights.  SE and TK were supported by the NSERC Discovery Grant program during the writing of this article.

\end{document}